\documentclass[11pt,a4paper]{article}
\usepackage[margin=2.5cm]{geometry}
\usepackage[T1]{fontenc}
\usepackage[utf8]{inputenc}
\usepackage{lmodern}
\usepackage{amsmath,amssymb,amsthm}
\usepackage{microtype}
\usepackage{enumitem}
\usepackage[hidelinks]{hyperref}
\hypersetup{
  pdftitle={Macroscopic Poincar\'e Inequalities and Weak-Type Characterizations of p-Energy Forms},
  pdfauthor={Fabrice Baudoin, Andrea Pinamonti, Meng Yang},
  pdfsubject={Analysis on metric spaces; p-energy forms; weak-type characterizations},
  pdfkeywords={p-energy forms, Poincar\'e inequalities, Korevaar--Schoen energy, weak monotonicity, weak-type characterizations},
  pdflang={en}
}

\numberwithin{equation}{section}
\theoremstyle{plain}
\newtheorem{theorem}{Theorem}[section]
\newtheorem{proposition}[theorem]{Proposition}
\newtheorem{lemma}[theorem]{Lemma}
\newtheorem{corollary}[theorem]{Corollary}
\theoremstyle{definition}
\newtheorem{assumption}[theorem]{Assumption}
\theoremstyle{remark}
\newtheorem{remark}[theorem]{Remark}

\newcommand{\E}{\mathcal E}
\newcommand{\F}{\mathcal F}
\newcommand{\Floc}{\mathcal F_{\mathrm{loc}}}
\newcommand{\Gp}{\Gamma}
\newcommand{\KS}{\mathrm{KS}}
\newcommand{\GpKS}{\Gamma^{\KS}}
\newcommand{\B}{\mathcal B}
\newcommand{\suppm}{\operatorname{supp}_{m}}
\newcommand{\capc}{\operatorname{cap}}
\newcommand{\diam}{\operatorname{diam}}

\newcommand{\Xint}[1]{\mathchoice
  {\XXint\displaystyle\textstyle{#1}}%
  {\XXint\textstyle\scriptstyle{#1}}%
  {\XXint\scriptstyle\scriptscriptstyle{#1}}%
  {\XXint\scriptscriptstyle\scriptscriptstyle{#1}}%
  \!\int}
\newcommand{\XXint}[3]{{\setbox0=\hbox{$#1{#2#3}{\int}$}%
  \vcenter{\hbox{$#2#3$}}\kern-.6\wd0}}
\newcommand{\avint}{\Xint-}

\title{Macroscopic Poincar\'e Inequalities \\  and
Weak-Type Characterizations of $p$-Energy Forms}
\author{Fabrice Baudoin\thanks{Department of Mathematics, Aarhus University,
Ny Munkegade 118, DK-8000 Aarhus C, Denmark.}
\and Andrea Pinamonti\thanks{Department of Mathematics, University of Trento,
Via Sommarive 14, 38123 Povo (Trento), Italy.}
\and Meng Yang\footnotemark[1]}
\date{}

\begin{document}

\maketitle

\begin{abstract}
For a class of regular strongly local $p$-energy forms on complete doubling
metric measure spaces, $1<p<\infty$, we derive a macroscopic Poincar\'e
inequality from a Poincar\'e inequality and a capacity upper bound at a
doubling scale~$\Psi$.  The macroscopic inequality compares
$L^p$-oscillations at any two scales and yields a uniform two-scale
comparison $J_s\lesssim J_r$, $0<r<s$, for the Korevaar--Schoen
functionals, strengthening their usual weak monotonicity property.
Conversely, under a density assumption on compactly supported functions,
the macroscopic inequality yields a regular strongly local
Korevaar--Schoen $p$-energy satisfying a $\Psi$-Poincar\'e inequality.
We also show that the macroscopic inequality forces lower $p$-scaling
of~$\Psi$ on scales not exceeding~$\diam X$; together with doubling, this
gives two-sided power bounds.  Finally, we establish a weak-type
characterization of the energy domain of Brezis-Van Schaftingen-Yung
type.
\end{abstract}

\tableofcontents

\section{Introduction}

Poincar\'e inequalities are a basic principle of first-order analysis:
the oscillation of a function on a ball is controlled by its energy near
that ball.  On metric measure spaces they provide a way to study Sobolev
functions without using coordinates or derivatives, see  \cite{HKST}.  A macroscopic
Poincar\'e inequality expresses a related principle using only
oscillations: variation at a given scale is controlled by averaged
variation at any smaller scale.  Hyt\"onen and Korte recently established such an
inequality on doubling Poincar\'e spaces, comparing an $L^1$-mean
oscillation with an $L^p$-average of smaller-scale oscillations, normalized
by the usual power scale $r^p$
\cite[Theorem~3.4]{HytonenKorte}.  This raises the question of whether the
same principle holds for energies and geometries whose natural scale is
different from $r^p$.

Regular strongly local $p$-energy forms provide a robust newly developed framework for this
question.  For $1<p<\infty$, such a form $(\E,\F)$ consists of a Sobolev
domain $\F$ and a $p$-energy $\E$; see
\cite{KajinoShimizuContraction,KajinoShimizuKS,Sasaya2026}.
This framework includes the usual gradient energies on Riemannian
manifolds and Carnot groups, as well as energies on fractals.  In the
classical examples the natural scale is $\Psi(r)=r^p$, whereas on fractals
it may be a different power or even a nonpower function.  Moreover, an
energy measure may be singular with respect to the underlying measure,
so that there need not be a gradient density to integrate.  These features
make it useful to formulate Poincar\'e inequalities directly in terms of
energy measures and a scale function $\Psi$.

Our main result is a macroscopic Poincar\'e inequality for $p$-energy forms
at general scales.  We work on a complete doubling metric measure space,
with a doubling scale function $\Psi$ and an energy form satisfying the
standing Assumption~\ref{ass:E} of Section~\ref{sec:framework}.  We assume the usual
$\Psi$-Poincar\'e inequality and a capacity upper bound, which says that
cutoff functions for a ball can be chosen with energy controlled by the
volume of the ball divided by its scale.  For $f\in L^p_{\mathrm{loc}}(X;m)$, write
\[
 M_f(x,r):=\frac1{\Psi(r)}\avint_{B(x,r)}|f-f_{B(x,r)}|^p\,dm.
\]
Theorem~\ref{thm:main} gives constants $C_M>0$ and $A_M\ge1$ such that
\[
 M_f(x,s)\le C_M\avint_{B(x,A_Ms)}M_f(y,r)\,dm(y),
 \qquad x\in X,\quad 0<r<s.
\]
Thus normalized oscillation at a given scale is controlled by its average
at every smaller scale, with constants independent of $f,r$, and~$s$.
For $\Psi(r)=r^p$ it is close in spirit to
\cite[Theorem~3.4]{HytonenKorte}, but the two statements are not directly
comparable: the left-hand side there is an $L^1$-mean oscillation,
whereas both sides of our inequality involve $L^p$-oscillations.

Thus an inequality initially involving the energy yields a comparison
that depends only on the metric, the measure, and the scale function.
In particular, the function need not belong to the energy domain for the
macroscopic inequality to apply.  The proof uses partitions of unity
constructed from capacitary cutoffs, allowing the argument to work with
singular energy measures.  Related partitions and discrete approximations
in $p$-Dirichlet spaces appear in
\cite[Lemmas~4.2 and~4.5]{ErikssonBique2026}.
Only doubling of $\Psi$ is assumed at this stage; no lower power bound is
required as an input.  A lower $p$-scaling bound is derived afterwards as
a consequence (Corollary~\ref{cor:forward-LS}).

Our first application is a uniform comparison of the Korevaar--Schoen functionals~\cite{KorevaarSchoen1993}.  For~$f\in L^p(X;m)$, set
\[
 J_r(f):=\frac1{\Psi(r)}\int_X\avint_{B(x,r)}
 |f(x)-f(y)|^p\,dm(y)\,dm(x).
\]
Integrating the macroscopic inequality gives the two-scale estimate
$J_s(f)\lesssim J_r(f)$ whenever $0<r<s$, with a constant independent of
both scales.  Consequently, control along any sequence of radii tending to
zero gives control at every radius.  This strengthens the weak
monotonicity property
\[
 \sup_{r>0}J_r(f)\lesssim\liminf_{r\downarrow0}J_r(f).
\]
A version of the latter for $p=1$ appears in the study of $BV$ functions
in \cite{ABCRST3}; for general $p$ it was formulated as property
$P(p,\alpha)$ in \cite[Definition~4.5]{BaudoinKSS}.  The terminology
weak monotonicity is used in \cite{GaoYuZhang2026} and
\cite[Definition~6.3.5]{KajinoShimizuKS}.
Related results on trees appear in \cite{BaudoinChenYangTrees}, and on
self-similar Laakso-type spaces in
\cite[Theorem~7.1]{AnttilaErikssonBiqueShimizu2026}.
To our knowledge, the available
results on fractals and in Dirichlet-form or $p$-energy settings, for
instance \cite{GaoYuZhang2026,ChangGaoYuZhang2024,ChenGaoYuZhang2026,
KajinoShimizuKS} and \cite[Theorem~1.4]{MuruganShimizuCarpet}, establish
weak monotonicity in this $\sup$--$\liminf$ form.  For $\Psi(r)=r^p$ on
doubling Poincar\'e spaces, a uniform two-scale comparison of the
corresponding $L^1$-oscillation functionals follows by integrating
\cite[Theorem~3.4]{HytonenKorte}; the two-scale comparison of the
Korevaar--Schoen functionals at general scales~$\Psi$ appears to be new.

A second application is geometric.  Proposition~\ref{prop:automatic-LS} shows that volume doubling and the macroscopic Poincar\'e inequality imply that~$X$ is connected and force the scale function to satisfy
\[
 \frac{\Psi(r)}{\Psi(R)}\lesssim\left(\frac rR\right)^p,
 \qquad 0<r\le R\le\diam X.
\]
When $X$ is unbounded, there is no upper restriction on $R$.
The proposition requires no energy form, chain condition, or doubling assumption on~$\Psi$.  It obtains the lower $p$-scaling restriction of~\cite{Yang2026} directly from an inequality involving only oscillations.  Together with doubling of~$\Psi$, this bound shows that every scale function admissible in our forward results is of two-sided power type (Remark~\ref{rem:O-regular}); such scale functions need not be powers, as the tree and Laakso-type examples show.  The lower bound will also be used in the weak-type characterization below.

Under our hypotheses, the domain characterization of Shimizu~\cite[Corollaries~1.14 and~1.15]{ShimizuBesov} gives
\[
 \F
 =\left\{f\in L^p(X;m):\liminf_{r\downarrow0}J_r(f)<\infty\right\}
 =\B_\Psi,
 \qquad
 \B_\Psi:=\left\{f\in L^p(X;m):\sup_{r>0}J_r(f)<\infty\right\},
\]
and, for $f\in\F$,
\[
 \E(f)\asymp\liminf_{r\downarrow0}J_r(f)
 \asymp\sup_{r>0}J_r(f).
\]
We recall this essentially known characterization without proof in Subsection~\ref{subsec:domain-id}, with comments on its applicability under the present hypotheses because t will be used in the reconstruction of the energy and in the weak-type characterization below. For connections with other Sobolev theories on metric spaces, see
\cite{KoskelaShanmugalingamTyson2004,AlonsoRuizBaudoin2025}.

The converse asks whether an energy can be constructed from the
macroscopic inequality itself.  Starting with a complete doubling metric
measure space and a doubling scale function satisfying this inequality,
we assume that $\B_\Psi\cap C_c(X)$ is uniformly dense in $C_c(X)$.
This assumption cannot be dropped: the macroscopic inequality persists
when $\Psi$ is multiplied by a nondecreasing function
(Remark~\ref{rem:mPI-monotone}), whereas Korevaar--Schoen classes at too
fast a scale contain only constants.
A common subsequential limit of the functionals $J_r$ then yields a
regular strongly local $p$-energy on the domain $\F_{\KS}$ obtained by
closing $\B_\Psi\cap C_c(X)$ in its natural Besov norm.  The energy is comparable
to $\sup_{r>0}J_r$ on this domain and satisfies a
$\Psi$-Poincar\'e inequality with a fixed dilation of balls.
The existence of the form follows from the Korevaar--Schoen theory of
$p$-energy forms \cite{KajinoShimizuKS,Sasaya2026} once weak monotonicity
is available; the new point is the Poincar\'e inequality, which we recover
from the macroscopic inequality.  For related constructions of
$p$-energies as limits of discrete or nonlocal energies, see
\cite{Kigami2023,CaoGuQiu2022,MuruganShimizuCarpet} and, for $p=2$,
\cite{KumagaiSturm2005}.
Under the forward hypotheses, $\F_{\KS}=\F$ and the constructed energy is
equivalent to the original one.
For a preassigned energy form, the converse takes a related form:
Poincar\'e is equivalent to the macroscopic inequality together with a
local bound on small-scale oscillations by that energy, under volume
doubling and the capacity upper bound
(Proposition~\ref{prop:microscopic-control}).

Our final result gives a second description of the same energy domain,
using the distribution of pairwise differences instead of their averages
at a fixed radius.  In Euclidean space, the
Brezis--Van Schaftingen--Yung formula recovers the Sobolev energy from the
weak $L^p$ size of the difference quotient
$|f(x)-f(y)|/|x-y|^{1+n/p}$
\cite{BVSY2021,BSSVY2024,DominguezMilman2022}.
Poliakovsky's characterization further shows that, for an $L^p$ function
on a Euclidean Lipschitz domain, the behavior of the distribution at high
levels already determines Sobolev membership
\cite[Theorem~1.3]{Poliakovsky2022}.
Metric versions for Lipschitz test functions were developed in
\cite{DLYYZ,DaiLiYangYuanZhao2026}, and
Hyt\"onen--Yang--Yuan--Zhao characterize the full homogeneous Sobolev
space defined by $p$-weak upper gradients on complete doubling $p$-Poincar\'e
spaces \cite[Theorem~1.1(ii)]{HytonenYangYuanZhao2026}.
Related nonlocal characterizations of Sobolev and $BV$ functions appear in
\cite{DS,LahtiPinamontiZhou20242,LahtiPinamontiZhou2024,BrenaPinamonti}.

In our setting, the scale and volume determine the appropriate quotient.
Writing $V(x,r)=m(B(x,r))$, set
\[
 K_{\Psi,f}(x,y):=
 \frac{|f(x)-f(y)|^p}
 {\Psi(d(x,y))V(x,d(x,y))},\qquad x\ne y,
\]
and set $K_{\Psi,f}=0$ on the diagonal.
Under the forward hypotheses, Theorem~\ref{thm:pairwise-weak} shows that
an $L^p$ function belongs to $\F$ if and only if either of the following
two quantities is finite; for $f\in\F$, both are equivalent to the energy:
\[
 \E(f)
 \asymp\sup_{\lambda>0}\lambda\,(m\otimes m)
       \bigl(\{K_{\Psi,f}>\lambda\}\bigr)
 \asymp\limsup_{\lambda\to\infty}\lambda\,(m\otimes m)
       \bigl(\{K_{\Psi,f}>\lambda\}\bigr).
\]
For $\Psi(r)=r^p$ and upper-gradient energies, the weak-norm
characterization is covered by the metric theory cited above.
Here the same description extends to doubling scale functions other than
$r^p$ and to possibly singular energy measures.

The link with the preceding results is the macroscopic control of
oscillations.  A bound on the high-level tail of $K_{\Psi,f}$ yields
control of localized Korevaar--Schoen functionals along small scales.
The macroscopic inequality propagates this control, and the domain
characterization then gives membership in $\F$.  In the other direction,
Poincar\'e estimates and a covering argument bound the weak quantity by
the energy, using the energy measure directly even when it is singular.
The Korevaar--Schoen and Brezis--Van Schaftingen--Yung descriptions thus
recover the same first-order structure from two different ways of
measuring pairwise differences.

\vspace{0.8cm}

Section~\ref{sec:framework} gives the standing assumptions, states the
main results, and presents examples.  Section~\ref{sec:macroscopic} proves
the macroscopic Poincar\'e inequality and its relation with microscopic
energy control, derives the two-scale comparison of the Korevaar--Schoen
functionals and lower $p$-scaling, and recalls the identification of the
energy domain.  Section~\ref{sec:converse} constructs a Korevaar--Schoen
$p$-energy from the macroscopic inequality.  Finally,
Section~\ref{sec:weak} proves the Brezis--Van Schaftingen--Yung-type
characterization.

\section{Standing framework and main results}\label{sec:framework}

In this section, we present the framework and our main results. Proofs are given in the later sections.

\subsection{Metric measure space and scale function}\label{sec:metric-framework}

Fix $p\in(1,\infty)$.  Let $(X,d)$ be a nontrivial complete metric space,
that is, a complete metric space with at least two points.
For $x\in X$ and $r>0$ write
\[
 B(x,r)=\{y\in X:d(x,y)<r\}.
\]
Let $m$ be a positive Radon measure on $X$ with full support such that
\[
 0<m(B(x,r))<\infty,\qquad x\in X,\ r>0.
\]
Set $V(x,r)=m(B(x,r))$.
If $B=B(x,r)$ and $\lambda>0$, write $\lambda B=B(x,\lambda r)$.  For a
measurable set $A$ with $0<m(A)<\infty$ put
\[
h_A= \avint_A h\,dm=\frac1{m(A)}\int_A h\,dm.
\]
Since $X=\bigcup_{n\ge1}B(o,n)$ for any $o\in X$, the measure $m$ is
$\sigma$-finite.

The volume doubling condition is
\begin{equation*}\label{eq:VD}\tag*{$\text{VD}$}
 V(x,2r)\le C_DV(x,r),\qquad x\in X,\ r>0,
\end{equation*}
for some $C_D\ge1$.

For $a\ge1$, set
\begin{equation}\label{eq:CD-fixed}
 C_D(a):=C_D^{\lceil\log_2 a\rceil}.
\end{equation}
It is well known that \eqref{eq:VD} implies
$V(x,ar)\le C_D(a)V(x,r)$ whenever $a\ge1$.

By the classical packing argument that we recall in Lemma~\ref{lem:packing}, \eqref{eq:VD}
implies that $(X,d)$ is metrically doubling.  Completeness then implies
properness.  In particular, $X$ is locally compact and separable, and every
ball has compact closure.  We use these consequences whenever
\eqref{eq:VD} is assumed.

Let $\Psi:[0,\infty)\to[0,\infty)$ be a doubling homeomorphism: there is
$C_\Psi\ge1$ such that
\begin{equation}\label{eq:Psi-doubling}
 \Psi(2r)\le C_\Psi\Psi(r),\qquad r>0.
\end{equation}
For $a\ge1$, set $C_\Psi(a):=C_\Psi^{\lceil\log_2a\rceil}$.  Then
\begin{equation}\label{eq:Psi-fixed}
 \Psi(ar)\le C_\Psi(a)\Psi(r),\qquad r>0.
\end{equation}
Throughout, $m\otimes m$ denotes the corresponding product measure.  For
nonnegative quantities $A$ and
$B$, the notation $A\lesssim B$ means $A\le CB$, $A\gtrsim B$ means
$B\lesssim A$, and $A\asymp B$ means that both inequalities hold.  Unless
otherwise stated, $C$ denotes a positive constant that may change from line
to line and depends only on the data explicitly named in the statement
under consideration.

For every positive $\sigma$-finite Radon measure $\mu$ and every Borel,
locally $\mu$-integrable function $h$, parameter-dependent integration
shows that the maps $(x,r)\mapsto\mu(B(x,r))$ and
$(x,r)\mapsto\int_{B(x,r)}h\,d\mu$ are Borel.  The same holds for ball
averages whenever their denominators are positive and finite.

For $u\in L^1_{\mathrm{loc}}(X;m)$ we use the representative-independent
measure-theoretic support
\begin{equation}\label{eq:suppm}
 \suppm u:=\operatorname{supp}(|u|m).
\end{equation}
Thus $u=0$ $m$-a.e. on $X\setminus\suppm u$.  This convention is useful
because the energy measures below need not be absolutely continuous with
respect to $m$.

\subsection{\texorpdfstring{$p$}{p}-Energy forms and analytic assumptions}
\label{sec:energy-framework}

Let $(\E,\F)$ be a $p$-energy form on $L^p(X;m)$ with domain
$\F\subset L^p(X;m)$.  The forward results are stated under the following
standing assumption.

\begin{assumption}\label{ass:E}
\begin{enumerate}[label=(E\arabic*)]
\item\label{E1} $(\E,\F)$ is a strongly local regular $p$-energy form in
the sense of Kajino--Shimizu and Sasaya
\cite{KajinoShimizuContraction,Sasaya2026}: $\F$ is dense in
$L^p(X;m)$, $\E^{1/p}$ is a seminorm, and the form is closed, Markovian,
regular, and strongly local.
\item\label{E2} The $p$-Clarkson inequalities hold: with $q=p/(p-1)$, for
every $f,g\in\F$,
\begin{equation}\label{eq:clarkson-low}
 \begin{aligned}
 2\bigl(\E(f)^{q/p}+\E(g)^{q/p}\bigr)^{p-1}
 &\le \E(f+g)+\E(f-g)\\
 &\le 2\bigl(\E(f)+\E(g)\bigr),
 \end{aligned}
 \qquad 1<p\le2,
\end{equation}
and
\begin{equation}\label{eq:clarkson-high}
 \begin{aligned}
 2\bigl(\E(f)+\E(g)\bigr)
 &\le \E(f+g)+\E(f-g)\\
 &\le 2\bigl(\E(f)^{q/p}+\E(g)^{q/p}\bigr)^{p-1},
 \end{aligned}
 \qquad 2\le p<\infty.
\end{equation}
\item\label{E3} $\F\cap L^\infty(X;m)$ is an algebra and
\begin{equation}\label{eq:algebra}
 \E(fg)^{1/p}
 \le \|f\|_{\infty}\E(g)^{1/p}+\|g\|_{\infty}\E(f)^{1/p},
 \qquad f,g\in\F\cap L^\infty(X;m).
\end{equation}
\item\label{E4} The canonical $p$-energy measures satisfy the localization
property \ref{G4} below.
\end{enumerate}
\end{assumption}

At $p=2$ the two Clarkson inequalities agree.  Under \ref{E1}, the
normal-contraction property
\begin{equation}\label{eq:normal-contraction}
 T\circ f\in\F,\qquad \E(T\circ f)\le \E(f)
\end{equation}
whenever $T:\mathbb R\to\mathbb R$ is $1$-Lipschitz, $T(0)=0$, and
$f\in\F$, follows from the energy-form axioms; see
Sasaya~\cite[Proposition~4.5]{Sasaya2026}.

For $f\in\F$, let $\Gp(f)$ denote its canonical $p$-energy measure, which
exists under \ref{E1} by Sasaya~\cite[Theorem~2.4]{Sasaya2026}.  We also
introduce the local domain
\begin{equation}\label{eq:Floc}
 \Floc=\left\{u\in L^p_{\mathrm{loc}}(X;m):\ \begin{aligned}
 &\text{for every relatively compact open }U\subset X,\\
 &\text{there is }u^\#\in\F\text{ with }u=u^\#\quad m\text{-a.e. on }U
 \end{aligned}\right\}.
\end{equation}
We use the following properties of the energy measures:
\begin{enumerate}[label=(G\arabic*)]
\item\label{G1} $\Gp(f)$ is a positive Radon measure and
$\Gp(f)(X)=\E(f)$ for every $f\in\F$;
\item\label{G2} for every Borel set $A$, the map
$f\mapsto \Gp(f)(A)^{1/p}$ is a seminorm on $\F$;
\item\label{G3} if $f,g\in\F\cap C_c(X)$ and $f-g$ is constant on a Borel
set $A$, then $\Gp(f)(A)=\Gp(g)(A)$;
\item\label{G4} if $f,g\in\F$ agree $m$-a.e. on an open set $U$, then
$\left.\Gp(f)\right|_U=\left.\Gp(g)\right|_U$.  Consequently the map
$f\mapsto\Gp(f)$ localizes consistently to $\Floc$.
\end{enumerate}
The global properties \ref{G1}--\ref{G3} follow from
Sasaya~\cite[Theorem~2.4]{Sasaya2026}.  The localization property
\ref{G4}, that is, \ref{E4}, is imposed separately as the local-domain
compatibility needed to define $\Gp(u)$ for $u\in\Floc$; compare
Yang~\cite[paragraph preceding Lemma~3.4]{Yang2026}.

The Poincar\'e inequality $\mathrm{PI}(\Psi)$ means that there are $C_P>0$
and $A_P\ge1$ such that
\begin{equation*}\label{eq:PI}\tag*{$\text{PI}(\Psi)$}
 \int_B|u-u_B|^p\,dm\le C_P\Psi(R)\Gp(u)(A_PB)
\end{equation*}
for every ball $B=B(x,R)$ and $u\in\F$.  For Borel sets $A_1,A_2\subset X$
define
\[
 \capc(A_1,A_2)
 =\inf\left\{\E(\phi):\begin{array}{l}
 \phi\in\F,\ \phi=1\ m\text{-a.e. on an open neighborhood of }A_1,\\
 \phi=0\ m\text{-a.e. on an open neighborhood of }A_2
 \end{array}\right\}.
\]
The capacity upper bound $\mathrm{cap}(\Psi)_{\le}$ means that there are
$C_c>0$ and $A_c>1$ such that
\begin{equation*}\label{eq:cap}\tag*{$\text{cap}(\Psi)_\le$}
 \capc\bigl(B(x,R),X\setminus B(x,A_cR)\bigr)
 \le C_c\frac{V(x,R)}{\Psi(R)},
 \qquad x\in X,\quad 0<R<\diam X.
\end{equation*}
Thus, when $X$ is bounded, no capacity condition at or above its diameter
is included in $\mathrm{cap}(\Psi)_{\le}$.
At an admissible radius for which $B(x,A_cR)=X$, the estimate is automatic.
Indeed, boundedness and the preceding properness observation make $X$
compact; the regularity and normal-contraction argument used in the proof of
Lemma~\ref{lem:constants} gives $1\in\F$, and \ref{G1} and \ref{G3} give
$\E(1)=0$.

When the corresponding assumptions are in force, the structural data are
\[
 p,\ C_D,\ C_\Psi,\ C_P,\ A_P,\ C_c,\ A_c.
\]
Unless otherwise specified, constants in the forward results depend only
on these data.

\subsection{Macroscopic inequality}

For $f\in L^p_{\mathrm{loc}}(X;m)$ set
\begin{equation}\label{eq:M}
 M_f(x,r)=\frac1{\Psi(r)}\avint_{B(x,r)}|f-f_{B(x,r)}|^p\,dm.
\end{equation}
The preceding measurability observation shows that $(x,r)\mapsto M_f(x,r)$
is Borel.  It is unchanged by altering $f$ on an $m$-null set, so the
averages below are independent of the representative.

We say that the macroscopic Poincar\'e inequality \eqref{eq:mPI} holds if there exist $C_M>0$ and $A_M\ge1$ such that for every $f\in L^p_{\mathrm{loc}}(X;m)$ and $x\in X$,
\begin{equation*}\label{eq:mPI}\tag*{$\text{mPI}(\Psi)$}
 M_f(x,s)\le C_M\avint_{B(x,A_Ms)}M_f(y,r)\,dm(y),\qquad 0<r<s.
\end{equation*}

\begin{theorem}\label{thm:main}
Suppose that Assumption~\ref{ass:E} holds, together with \eqref{eq:VD},
\eqref{eq:PI}, and \eqref{eq:cap}.  Then \eqref{eq:mPI} holds, with
constants depending only on the structural data.
\end{theorem}

\subsection{Consequences of the macroscopic inequality}

Define the Korevaar--Schoen functional
\begin{equation}\label{eq:J}
 J_r(f):=\frac1{\Psi(r)}\int_X\avint_{B(x,r)}|f(x)-f(y)|^p\,dm(y)\,dm(x),
\end{equation}
and the Besov--Korevaar--Schoen space
\begin{equation}\label{eq:Bpsi}
 \begin{aligned}
 \B_\Psi&:=\{f\in L^p(X;m):\sup_{r>0}J_r(f)<\infty\},\\
 \|f\|_{\B_\Psi}&:=\|f\|_{L^p}+\bigl(\sup_{r>0}J_r(f)\bigr)^{1/p}.
 \end{aligned}
\end{equation}

The macroscopic inequality integrates to a uniform two-scale comparison of
the Korevaar--Schoen functionals; the proof is given in
Subsection~\ref{sec:two-scale}.

\begin{theorem}\label{thm:two-scale}
Assume \eqref{eq:VD}, \eqref{eq:Psi-doubling}, and \eqref{eq:mPI}.  Then
for every $f\in L^p(X;m)$ and $0<r<s$,
\begin{equation}\label{eq:J-WM}
 J_s(f)\le CJ_r(f),
\end{equation}
where $C$ depends only on $p,C_D,C_\Psi,C_M$, and $A_M$.  In particular,
\begin{equation}\label{eq:WM}
 \sup_{s>0}J_s(f)\le C\liminf_{r\downarrow0}J_r(f).
\end{equation}
\end{theorem}

The macroscopic inequality also has a geometric consequence, which
requires neither an energy form nor doubling of~$\Psi$.  For $\beta>0$,
we say that $\Psi$ satisfies $\mathrm{LS}(\beta)$ if there exists
$C_L\ge1$ such that
\begin{equation}\label{eq:Psi-lower}
 \frac{\Psi(r)}{\Psi(R)}
 \le C_L\left(\frac rR\right)^\beta,
 \qquad 0<r\le R<\infty.
\end{equation}
When $\diam X<\infty$, the condition is required only for
$R\le\diam X$.  The proof of the following proposition is given in
Subsection~\ref{sec:scaling}.

\begin{proposition}
\label{prop:automatic-LS}
Let $p\in(1,\infty)$, let $(X,d)$ be a nontrivial complete metric space,
and let $m$ be a positive Radon measure of full support satisfying
\[
 0<m(B(x,r))<\infty,\qquad x\in X,\ r>0.
\]
Let $\Psi:[0,\infty)\to[0,\infty)$ be an increasing homeomorphism.
Assume
\eqref{eq:VD} and \eqref{eq:mPI}.  Then $X$ is connected and $\Psi$
satisfies $\mathrm{LS}(p)$:
\begin{equation}\label{eq:automatic-LS}
 \frac{\Psi(r)}{\Psi(R)}
 \le C\left(\frac rR\right)^p
\end{equation}
whenever $0<r\le R<\infty$ and, if $\diam X<\infty$,
$R\le\diam X$.  The constant depends only on $p,C_D,C_M$, and $A_M$.
\end{proposition}

Under the hypotheses of Theorem~\ref{thm:main}, the domain of the given energy form is identified by~\eqref{eq:domain}, with the energy comparison~\eqref{eq:energy-equiv}.  This is the characterization of Shimizu~\cite[Corollaries~1.14 and~1.15]{ShimizuBesov}; its applicability under the present hypotheses is discussed in Subsection~\ref{subsec:domain-id}.

\subsection{Converse construction}

We now assume only the metric-measure hypotheses and construct an energy
from \eqref{eq:mPI}.  The construction is subsequential and does not assert
uniqueness of the resulting form.

\begin{theorem}\label{thm:converse}
Assume the standing metric-measure hypotheses, \eqref{eq:VD},
\eqref{eq:Psi-doubling}, and \eqref{eq:mPI}.  Fix admissible
constants $C_M$ and $A_M$ in \eqref{eq:mPI}.
Suppose in addition that
\begin{equation}\label{eq:Core}
 \B_\Psi\cap C_c(X)\quad\text{is dense in }(C_c(X),\|\cdot\|_\infty).
\end{equation}
Then there exist a sequence $r_n\downarrow0$, a functional
$\widetilde{\E}_{\KS}$ on $\B_\Psi$, and a regular strongly local
$p$-energy form $(\E_{\KS},\F_{\KS})$ such that
\begin{equation}\label{eq:KS-energy}
 \widetilde{\E}_{\KS}(f)=\lim_{n\to\infty}J_{r_n}(f)
 \asymp\sup_{r>0}J_r(f),\qquad f\in\B_\Psi,
\end{equation}
and
\[
 \F_{\KS}=\overline{\B_\Psi\cap C_c(X)}^{\,\|\cdot\|_{\B_\Psi}}.
\]
Here $\E_{\KS}$ is the restriction of $\widetilde{\E}_{\KS}$ to
$\F_{\KS}$.  The form satisfies the full $p$-Clarkson inequalities and the generalized
$p$-contraction property.  Its canonical
$p$-energy measure is denoted by $\GpKS(u)$ for $u\in\F_{\KS}$.  All
constants in this theorem depend only on
\[
 p,\ C_D,\ C_\Psi,\ C_M,\ A_M.
\]
Moreover, with $A:=A_M+2$,
\begin{equation}\label{eq:converse-PI}
 \int_{B(x,R)}|u-u_{B(x,R)}|^p\,dm
 \le C\Psi(R)\GpKS(u)(B(x,AR))
\end{equation}
for every $u\in\F_{\KS}$, $x\in X$, and $R>0$.
\end{theorem}

The existence of the form follows from Sasaya's Korevaar--Schoen
construction \cite[Theorem~8.3]{Sasaya2026}, based on
\cite[Theorems~6.3.8 and~6.4.2]{KajinoShimizuKS}.
The new conclusion is the Poincar\'e inequality~\eqref{eq:converse-PI},
which we recover from~\eqref{eq:mPI} in Section~\ref{sec:converse}.
No capacity upper bound $\mathrm{cap}(\Psi)_{\le}$ is asserted for the
reconstructed form; see Remark~\ref{rem:converse-capacity}.

\subsection{Weak-type characterizations}

For $f\in L^p(X;m)$ choose a Borel representative and define, off the
diagonal,
\begin{equation}\label{eq:pairwise-quotients}
 K_{\Psi,f}(x,y)
 :=\frac{|f(x)-f(y)|^p}
 {\Psi(d(x,y))V(x,d(x,y))},
\end{equation}
and set $K_{\Psi,f}=0$ on the diagonal.  Its $m\otimes m$-equivalence class and the quantities below are independent
of the representative: a change on an $m$-null set $N$ affects only
$(N\times X)\cup(X\times N)$.  Put
\begin{align}
 \mathfrak W_\Psi(f)
 &:=\|K_{\Psi,f}\|_{L^{1,\infty}(m\otimes m)}
 =\sup_{\lambda>0}\lambda\,
 (m\otimes m)(\{K_{\Psi,f}>\lambda\}),                 \label{eq:pairwise-W}\\
 \mathfrak W_{\Psi,\infty}(f)
 &:=\limsup_{\lambda\to\infty}\lambda\,
 (m\otimes m)(\{K_{\Psi,f}>\lambda\}).                 \label{eq:pairwise-tail}
\end{align}
Under \eqref{eq:VD}, if $R=d(x,y)>0$, then
$C_D^{-1}V(x,R)\le V(y,R)\le C_DV(x,R)$.  Thus replacing the
first-coordinate volume in \eqref{eq:pairwise-quotients} by the
second-coordinate volume changes the functionals by factors depending only
on $C_D$.

\begin{theorem}\label{thm:pairwise-weak}
Suppose that Assumption~\ref{ass:E}, \eqref{eq:VD}, \eqref{eq:PI}, and
\eqref{eq:cap} hold.  Then, for every $f\in L^p(X;m)$,
\begin{equation}\label{eq:pairwise-domain}
 f\in\F
 \quad\Longleftrightarrow\quad
 \mathfrak W_{\Psi,\infty}(f)<\infty
 \quad\Longleftrightarrow\quad
 \mathfrak W_\Psi(f)<\infty.
\end{equation}
Moreover, for $f\in\F$,
\begin{equation}\label{eq:pairwise-energy}
 \E(f)\asymp
 \mathfrak W_{\Psi,\infty}(f)
 \asymp\mathfrak W_\Psi(f).
\end{equation}
The constants depend only on the structural data.
\end{theorem}

\begin{remark}\label{rem:pairwise-regimes}
Under the hypotheses of Theorem~\ref{thm:pairwise-weak}, suppose that
$V(x,r)\asymp r^Q$ and $\Psi(r)\asymp r^\alpha$ for
$0<r\le\diam X$, with the upper restriction omitted when
$\diam X=\infty$.  Then
Corollary~\ref{cor:forward-LS} forces $\alpha\ge p$.
For $X=\mathbb R^Q$ and
$\alpha=p$, \eqref{eq:pairwise-energy} is the Brezis--Van Schaftingen--Yung formula, up to the
normalization of ball volume \cite{BVSY2021,BSSVY2024}.  See also
\cite[Theorem~1.3]{DominguezEtAl2023} for Besov--Sobolev
characterizations, including the weak-type endpoint with Lorentz index~$\infty$.
\end{remark}

\subsection{Examples}\label{sec:examples}

We now describe several classes of spaces satisfying Assumption~\ref{ass:E}
together with \eqref{eq:VD}, \eqref{eq:PI}, and \eqref{eq:cap}.
In each case, the scale function is understood up to multiplication by
 constants.

For the Riemannian, Carnot, tree, and Laakso energies below, linearity of
the gradient operators and the classical scalar or Hilbert-valued $L^p$
Clarkson inequalities give~\ref{E2}.
The chain and Leibniz rules give normal contractions and~\ref{E3},
while the energy-measure formulas give~\ref{G1}--\ref{G2}.  Locality,
regularity, and strong locality are justified in the individual examples.

\paragraph{Riemannian and sub-Riemannian spaces.}
Let $(X,d,m)$ be complete, essentially nonbranching and infinitesimally Hilbertian, and satisfy the measure contraction property $\mathrm{MCP}(0,N)$ for some $1<N<\infty$. For $1<p<\infty$, set
\[
 \F=W^{1,p}(X),\qquad
 \E_p(f)=\int_X|Df|^p\,dm,\qquad
 \Gp(f)=|Df|^p m,
\]
where $|Df|$ denotes the minimal $p$-weak upper gradient. The MCP condition gives volume doubling and, under essential nonbranching, a scale-invariant weak $(1,1)$-Poincar\'e inequality; see~\cite[Section~7]{HanMilman2021}. Hence~\eqref{eq:PI} holds with $\Psi(r)=r^p$, and Lipschitz distance cutoffs give~\eqref{eq:cap}. The weak gradient is independent of $p$, and the Hilbertian differential gives the Clarkson inequalities~\cite[Theorem~3.3 and Proposition~4.4]{GigliNobili2022}; Sobolev locality and Lipschitz density give the remaining parts of Assumption~\ref{ass:E}~\cite{HKST}. This class includes complete Riemannian manifolds with nonnegative Ricci curvature and Riemannian volume, and complete Sasakian manifolds with nonnegative Tanaka--Webster Ricci curvature, equipped with the Carnot--Carath\'eodory metric and Popp measure~\cite[Theorem~1.1]{Lee2015}, including Heisenberg groups and standard CR spheres~\cite{LeeLiZelenko2016}. Sub-Riemannian manifolds are infinitesimally Hilbertian~\cite{LeDonneLucicPasqualetto2023}, and $|Df|$ is the horizontal gradient norm for smooth reference measures. The same standing assumptions also hold on complete doubling Riemannian manifolds supporting the scale-invariant $L^p$-Poincar\'e inequality, including compact Riemannian manifolds~\cite{SaloffCoste2001}, and on arbitrary Carnot groups with their Carnot--Carath\'eodory metric and Haar measure~\cite{Jerison1986}.

\paragraph{Post-critically finite self-similar sets.}
Let $\mathcal L=(K,S,\{F_i\}_{i\in S})$ be a connected p.c.f. self-similar
structure carrying a regular strongly local self-similar $p$-resistance form
$(\E_p,\F_p)$ with weights $(\rho_{p,i})_{i\in S}\in(1,\infty)^S$.  Set
\[
 \widehat R_p(x,y)=R_{\E_p}(x,y)^{1/(p-1)},
 \qquad
 \sum_{i\in S}\rho_{p,i}^{-d_{f,p}/(p-1)}=1,
\]
where $R_{\E_p}$ denotes the effective $p$-resistance, and let $m$ be the
self-similar probability measure with weights
$\rho_{p,i}^{-d_{f,p}/(p-1)}$.  Then $m$ is $d_{f,p}$-Ahlfors regular with
respect to $\widehat R_p$, and the standing Poincar\'e and capacity
assumptions hold with
\[
 \Psi(r)=r^{d_{f,p}+p-1},
\]
as follows from \cite[Assumption~6.6.1, Lemma~6.6.8, Corollary~6.6.11,
and Propositions~6.6.9, 6.6.12, and~6.6.14]{KajinoShimizuKS}.  The associated
canonical energy measures satisfy \ref{G1}--\ref{G3}, while the open-set
locality property \ref{G4} follows from
\cite[Lemma~5.19]{KajinoShimizuContraction}.  The generalized
$p$-contraction property yields the remaining Clarkson, normal-contraction,
and algebra properties by
\cite[Proposition~2.3]{KajinoShimizuContraction}.  For the cited self-similar
constructions, strong locality follows from
\cite[Proposition~5.24]{KajinoShimizuContraction}; in the Cao--Gu--Qiu
construction it is proved in \cite[Theorem~5.2]{CaoGuQiu2022}.  Such
forms exist on strongly symmetric p.c.f. structures, including affine
nested fractals and the standard Sierpi\'nski gasket; see
\cite[Theorem~8.43, Proposition~8.49, and Theorems~8.50--8.51]{KajinoShimizuContraction}
and
\cite[Theorems~4.2, 5.1, 5.2, and~6.3]{CaoGuQiu2022}.  This statement uses the $p$-resistance metric and the corresponding
self-similar measure.  It does not assert the same power scale for an
arbitrary ambient Euclidean metric.

\paragraph{The planar Sierpi\'nski carpet.}
For every $p\in(1,\infty)$, Murugan and Shimizu construct on the standard
planar carpet, with its Euclidean metric and uniform self-similar measure,
a regular strongly local self-similar $p$-energy and its energy measures,
which satisfy \ref{G1}--\ref{G3}; see
\cite[Theorems~1.1 and~1.2]{MuruganShimizuCarpet}, building on the
construction of $p$-energies and energy measures on carpets in
\cite{Shimizu2024}.  If $\rho(p)$ is
their renormalization constant, then the standing Poincar\'e and capacity
estimates hold with
\[
 \Psi(r)=r^{d_w(p)},
 \qquad d_w(p)=\frac{\log(8\rho(p))}{\log 3};
\]
see \cite[Theorem~1.5]{MuruganShimizuCarpet}.  The generalized
$p$-contraction property, and hence the required Clarkson,
normal-contraction, and algebra properties, is established in
\cite[Remark~8.20]{KajinoShimizuContraction}; the open-set locality
property \ref{G4} follows from
\cite[Lemma~5.19]{KajinoShimizuContraction}.

\paragraph{Metric trees and cable systems.}
Let $(T,d)$ be a complete real tree, let $\nu$ be its length measure, and let
$m$ be a full-support Radon measure such that
\[
 m(B(x,r))\asymp\Phi(r),\qquad x\in T,\ r>0,
\]
where $\Phi$ is a doubling homeomorphism satisfying, for some $\alpha\ge1$,
\[
 c\frac Rr\le\frac{\Phi(R)}{\Phi(r)}
 \le C\left(\frac Rr\right)^\alpha,
 \qquad 0<r\le R.
\]
Let $\F_p$ be the space of functions in $L^p(T;m)$ that are absolutely
continuous along compact geodesic arcs and whose arc derivative $\partial f$
belongs to $L^p(T;\nu)$, and consider the length-gradient energy
\[
 \E_p(f)=\int_T|\partial f|^p\,d\nu,
 \qquad \Gp(f)=|\partial f|^p\nu.
\]
Then the standing assumptions hold with
\[
 \Psi(r)=r^{p-1}\Phi(r).
\]
Locality of arc derivatives gives~\ref{G3}--\ref{G4}; compactly supported
arcwise absolutely continuous functions provide a regular core, and the
energy is strongly local.
The geodesic Morrey estimate gives $\mathrm{PI}(\Psi)$, while tree-adapted
cutoffs have energy $\lesssim r^{1-p}\asymp V(x,r)/\Psi(r)$; see
\cite[Lemma~2.10, Definitions~2.11 and~2.13, Theorem~2.12,
Propositions~2.14--2.15, Corollary~2.16, Lemma~3.3, Theorem~3.5, and
Proposition~3.6]{BaudoinChenYangTrees}.  The global volume setting is
treated in \cite[Section~5]{BaudoinChenYangTrees}.  Yang's scale-irregular
real trees give explicit unbounded instances with
$V(x,r)\asymp V_b(r)$ and $\Psi(r)=r^{p-1}V_b(r)$; see
\cite[Section~4, the paragraph preceding Lemma~4.1, Lemma~4.2, and
Propositions~4.3--4.4]{Yang2026}.  Finite connected
metric graphs, regarded as cable systems with length measure and edgewise
$p$-energy, give compact examples with $\Psi(r)=r^p$.  The same edgewise
calculus verifies Assumption~\ref{ass:E}.  Discrete edge forms
on simplicial trees are not strongly local and hence do not satisfy the
standing assumptions; the cable realization is the appropriate model in this setting.

\paragraph{Laakso-type spaces with prescribed scale functions.}
Let $\Phi$ and $\Psi$ be doubling
homeomorphisms satisfying
\[
 C^{-1}\left(\frac Rr\right)^p
 \le\frac{\Psi(R)}{\Psi(r)}
 \le C\left(\frac Rr\right)^{p-1}\frac{\Phi(R)}{\Phi(r)},
 \qquad 0<r\le R.
\]

Yang~\cite[Theorem~2.3]{Yang2026} constructs an unbounded geodesic
Laakso-type metric measure space with $V(x,r)\asymp\Phi(r)$ and a regular
strongly local $p$-energy satisfying $\mathrm{PI}(\Psi)$ and the cutoff
Sobolev inequality at scale $\Psi$; the latter implies
$\mathrm{cap}(\Psi)_{\le}$.  These constructions also allow scale functions other than powers.  By \cite[Lemma~6.1]{Yang2026}, the form and its energy measure
have the gradient representations
$\E_p(f)=\int|\nabla_{\mathcal L}f|^p\,d\lambda$ and
$\Gp(f)=|\nabla_{\mathcal L}f|^p\lambda$, with $\lambda$ and
$\nabla_{\mathcal L}$ as in that construction.  The local-domain compatibility in
\cite[paragraph preceding Lemma~3.4]{Yang2026} gives~\ref{G4}.  In the
Ahlfors-regular specialization
$\Phi(r)=r^{d_h}$ and $\Psi(r)=r^{\beta_p}$, every exponent in the sharp
range $p\le\beta_p\le d_h+p-1$ occurs
\cite[Theorem~2.3 and Corollary~2.6]{Yang2026}.
Self-similar Laakso-type energies are also constructed in
\cite{AnttilaErikssonBiqueShimizu2026}; those constructions are distinct
from the prescribed-scale spaces used here.

\section{Macroscopic inequalities and their consequences}\label{sec:macroscopic}

We first prove Theorem~\ref{thm:main} by capacitary discrete approximation.  We then prove the two-scale comparison of the Korevaar--Schoen functionals (Theorem~\ref{thm:two-scale}) and the connectedness and lower scaling of~$\Psi$ (Proposition~\ref{prop:automatic-LS}), before recalling the essentially known identification of the energy domain.

\subsection{Packing and localization}\label{sec:prelim}

We first recall the following lemma.

\begin{lemma}\label{lem:packing}
Assume \eqref{eq:VD}.  For every $\lambda\ge1$ there exists
$N_D(\lambda)<\infty$, depending only on $C_D$ and $\lambda$, with the
following property.  If $\{x_i\}_{i\in I}$ is $t$-separated, i.e.
$d(x_i,x_j)\ge t$ for $i\ne j$, then
\[
 \#\{i:x_i\in B(x,\lambda t)\}\le N_D(\lambda)
\]
for every $x\in X$ and $t>0$.
\end{lemma}

\begin{proof}
The balls $B(x_i,t/2)$ are disjoint and lie in
$B(x,(\lambda+1/2)t)$.  Iterated doubling compares the measure of this
ball with that of every $B(x_i,t/2)$, uniformly in $i$; summing the
disjoint measures proves the claim.  This is the standard packing
argument; see, e.g., \cite{Heinonen2001,HKST}.
\end{proof}

For $u\in\Floc$ and a relatively compact open set $U$, define
\begin{equation}\label{eq:local-Gamma}
 \left.\Gp(u)\right|_U:=\left.\Gp(u^\#)\right|_U,
\end{equation}
where $u^\#\in\F$ is any representative agreeing with $u$ on $U$.
Property \ref{G4} makes this definition independent of the representative
and yields a positive Radon measure on $X$.  This localization is also used for locally finite sums of cutoffs in
Yang~\cite[paragraph preceding Lemma~3.4]{Yang2026}.

We first record invariance under addition of constants.

\begin{lemma}\label{lem:constants}
Let $u\in\Floc$, let $U\Subset X$ be open, and let $c\in\mathbb R$.  Then
\[
 \left.\Gp(u+c)\right|_U=\left.\Gp(u)\right|_U.
\]
\end{lemma}

\begin{proof}
Choose $u^\#\in\F$ agreeing with $u$ near $\overline U$.  Regularity,
together with normal contraction, provides
$\chi\in\F\cap C_c(X)$ with $\chi=1$ near $\overline U$: choose
$h_0\in C_c(X)$ with $h_0=1$ near $\overline U$ and $0\le h_0\le1$,
then use regularity to find $h\in\F\cap C_c(X)$ with
$\|h-h_0\|_\infty<1/4$.  If $T$ is zero on $(-\infty,1/4]$, linear on
$[1/4,3/4]$, and one on $[3/4,\infty)$, then $T/2$ is a normal
contraction.  Hence $\chi:=T\circ h\in\F\cap C_c(X)$ by
\eqref{eq:normal-contraction} and homogeneity, and $\chi=1$ near
$\overline U$.  The function $u^\#+c\chi$ represents $u+c$ on $U$.  For every Borel $A\subset U$, \ref{G3} gives
$\Gp(c\chi)(A)=0$, and the two triangle inequalities furnished by
\ref{G2} yield
$\Gp(u^\#+c\chi)(A)^{1/p}=\Gp(u^\#)(A)^{1/p}$.  Hence the restrictions
coincide.
\end{proof}

\begin{lemma}\label{lem:PI-local}
Under \eqref{eq:PI}, the same inequality holds for $u\in\Floc$:
for every ball $B=B(x,R)$,
\[
 \int_B|u-u_B|^p\,dm
 \le C_P\Psi(R)\Gp(u)(A_PB).
\]
\end{lemma}

\begin{proof}
Because all balls are relatively compact, choose $u^\#\in\F$ agreeing
with $u$ on a relatively compact neighborhood of $A_PB$.  Apply
\eqref{eq:PI} to $u^\#$ and use the definition
\eqref{eq:local-Gamma}.
\end{proof}

\subsection{A controlled partition of unity}\label{sec:partition}

The next lemma supplies the capacitary partition used in the proof.  Controlled partitions of this
kind were used by Murugan~\cite[proof of Lemma~2.5]{Murugan2020},
Baudoin~\cite[proof of Theorem~5.2]{BaudoinKSS}, and
Shimizu~\cite[Lemma~3.10]{ShimizuBesov}; see also
Yang~\cite[Lemma~3.2]{Yang2026} and
Eriksson-Bique~\cite[Lemma~4.2]{ErikssonBique2026}.  For Lipschitz
partitions of unity subordinate to separated nets in doubling spaces, see
Hyt\"onen--Korte~\cite[Lemma~3.1]{HytonenKorte}.
We give the construction under the present assumptions, using the
normal-contraction property \eqref{eq:normal-contraction} and the algebra
estimate \eqref{eq:algebra}.

Note that since $X$ is separable, for every $t>0$ there is a countable maximal
$t$-separated set $\{x_i\}_{i\in I}$.  By maximality,
$X=\bigcup_i B(x_i,t)$, and by Lemma~\ref{lem:packing} every
fixed enlargement $\{B(x_i,\lambda t)\}_i$ has uniformly bounded
multiplicity.  In particular, such families are locally finite.

\begin{lemma}\label{lem:partition}
Assume \eqref{eq:VD} and \eqref{eq:cap}.  Set
\[
 L:=2A_c.
\]
There is $C_{\mathrm{cut}}<\infty$, depending only on
$p,C_D,C_\Psi,C_c,A_c$, such that the following holds.
For every $t>0$ satisfying
\[
 Lt<\frac12\diam X\qquad\text{if }\diam X<\infty,
\]
and every maximal $t$-separated set $\{x_i\}_{i\in I}$, there are functions
$\psi_i\in\F\cap L^\infty(X;m)$ such that
\begin{align}
 &0\le\psi_i\le1,\label{eq:part-pos}\\
 &\suppm\psi_i\subset B(x_i,Lt),\label{eq:part-supp}\\
 &\sum_{i\in I}\psi_i=1\quad m\text{-a.e. on }X,\label{eq:part-sum}\\
 &\E(\psi_i)\le C_{\mathrm{cut}}\frac{V(x_i,t)}{\Psi(t)}.
 \label{eq:part-energy}
\end{align}
The sum in \eqref{eq:part-sum} is locally finite, and every family
$\{B(x_i,\lambda t)\}_i$, for fixed $\lambda$, has uniformly bounded
multiplicity.
\end{lemma}

\begin{proof}
The scale restriction gives
$2t<\diam X$, so \eqref{eq:cap} is available at radius $2t$; moreover,
$B(x_i,Lt)$ is a proper subset of $X$.  This is immediate when
$\diam X=\infty$; when $\diam X<\infty$, it follows from
$\sup_{y\in X}d(x_i,y)\ge\frac12\diam X>Lt$.  Choose an admissible
near-minimizer at radius $2t$ and truncate it by normal contractions.
Doubling then gives $\eta_i\in\F$, $0\le\eta_i\le1$, such that
\[
 \eta_i=1\quad m\text{-a.e. on }B(x_i,2t),\qquad
 \suppm\eta_i\subset B(x_i,Lt),\qquad
 \E(\eta_i)\lesssim\frac{V(x_i,t)}{\Psi(t)}.
\]
The net is countable.  After changing representatives on one common null
set, we may assume that every $\eta_i$ vanishes off $\suppm\eta_i$ and has
the stated value on its inner ball.
Put $J_i=\{j:d(x_i,x_j)<2Lt\}$ and
$S_i=\sum_{j\in J_i}\eta_j$.  Lemma~\ref{lem:packing}, nearby-ball
volume comparison, and the triangle inequality for $\E^{1/p}$ give
\[
 \#J_i\lesssim1,\qquad
 \E(S_i)^{1/p}\lesssim
 \left(\frac{V(x_i,t)}{\Psi(t)}\right)^{1/p}.
\]
Define the $1$-Lipschitz normal contraction $H:\mathbb R\to[0,1]$ by
\[
 H(s)=
 \begin{cases}
 0,&s\le0,\\
 s,&0\le s\le1,\\
 s^{-1},&s\ge1.
 \end{cases}
\]
Set $\psi_i=\eta_iH(S_i)$.  Normal contraction and
\eqref{eq:algebra} give
\[
 \E(\psi_i)^{1/p}
 \lesssim \E(S_i)^{1/p}+\E(\eta_i)^{1/p}
 \lesssim\left(\frac{V(x_i,t)}{\Psi(t)}\right)^{1/p}.
\]
The family $\{\eta_i\}$ is locally finite.  Hence
\[
 S(x):=\sum_i\eta_i(x)
\]
is defined off the common null set and satisfies $S\ge1$, because the
$B(x_i,t)$ cover $X$.  On $\{\eta_i\ne0\}$ every nonzero $\eta_j$ has
$j\in J_i$;
therefore $S_i=S$ and $\psi_i=\eta_i/S$.  Summation proves
\eqref{eq:part-sum}; the other assertions follow from the support bounds
and Lemma~\ref{lem:packing}.
\end{proof}

\subsection{Discrete approximation at a fixed scale}\label{sec:approximation}

Fix $f\in L^p_{\mathrm{loc}}(X;m)$ and $0<r<s$.  If $\diam X<\infty$ and $r>\diam X$, then every ball $B(y,r)$ and $B(x,s)$ is $X$, so
\[
 M_f(x,s)=\frac{\Psi(r)}{\Psi(s)}M_f(y,r)\le M_f(y,r)
\]
for all $x,y$.  We therefore construct the approximation below whenever
$r\le\diam X$.

With $L=2A_c$ as in Lemma~\ref{lem:partition}, let
\begin{equation}\label{eq:L-theta}
 \theta:=\frac{1}{16(L+2)},
 \qquad
 t:=\theta r.
\end{equation}
If $\diam X<\infty$, then
\[
 2A_ct=Lt=L\theta r<\frac r{16}
 \le\frac1{16}\diam X<\frac12\diam X.
\]
Hence the scale hypothesis of Lemma~\ref{lem:partition} is satisfied.
Choose a maximal $t$-separated set $\{x_i\}$ and let
$\{\psi_i\}$ be the partition furnished by
Lemma~\ref{lem:partition}.  Put
\begin{equation}\label{eq:B-D-ai}
 B_i:=B(x_i,t),
 \qquad
 D_i:=B(x_i,r/4),
 \qquad
 a_i:=f_{B_i}.
\end{equation}
Define the locally finite discrete approximation
\begin{equation}\label{eq:Pr}
 P_rf:=\sum_i a_i\psi_i.
\end{equation}
Because the sum is locally finite and every $\psi_i\in\F$,
$P_rf\in\Floc$.

This is the capacitary counterpart of formula~(3.2) in
Hyt\"onen--Korte~\cite{HytonenKorte}.  Related discrete approximations appear
in Shimizu~\cite[Lemma~3.11]{ShimizuBesov} and
Yang~\cite[Section~3]{Yang2026}; compare also the discrete convolutions in
Eriksson-Bique~\cite[proof of Lemma~4.5]{ErikssonBique2026}.
The coefficient estimate below parallels the one in the proof of
Hyt\"onen--Korte~\cite[Lemma~3.3]{HytonenKorte}, with localized energy
measures in place of pointwise Lipschitz constants.

The choice of $\theta$ gives
\begin{equation}\label{eq:theta-geometry}
 \frac14+(L+2)\theta
 =\frac14+\frac1{16}
 =\frac5{16}<1.
\end{equation}

In this subsection, constants depend only on
$p,C_D,C_\Psi,C_c$, and $A_c$, unless an additional dependence is
specified.  The constants in Lemma~\ref{lem:energy} may also depend on
$A_P$.

\begin{lemma}\label{lem:coeff}
Let
\[
 K_i:=\{j:B_i\cap\suppm\psi_j\ne\varnothing\}.
\]
Then $\#K_i\le N_1$ uniformly in $i$, and for every $j\in K_i$,
\begin{equation}\label{eq:coeff-energy}
 |a_i-a_j|^p\E(\psi_j)
 \le C\int_{D_i}M_f(y,r)\,dm(y),
\end{equation}
where $N_1$ and $C$ depend only on
$p,C_D,C_\Psi,C_c$, and $A_c$.
\end{lemma}

\begin{proof}
If $j\in K_i$, choose $z\in B_i\cap\suppm\psi_j$.  By
\eqref{eq:part-supp},
\[
 d(x_i,x_j)
 \le d(x_i,z)+d(z,x_j)
 <(L+1)t.
\]
Lemma~\ref{lem:packing} therefore gives a uniform bound on $\#K_i$.

Fix $j\in K_i$ and $y\in D_i$, and set $Q_y=B(y,r)$.  If
$z\in B_j$, then by \eqref{eq:theta-geometry}
\[
 d(z,y)
 \le d(z,x_j)+d(x_j,x_i)+d(x_i,y)
 <t+(L+1)t+\frac r4
 = (L+2)t+\frac r4<r.
\]
Thus $B_j\subset Q_y$.  The same argument, with $j=i$, gives
$B_i\subset Q_y$.

By Jensen's inequality,
\begin{align*}
 |a_i-f_{Q_y}|^p
 &\le \frac1{V(x_i,t)}\int_{B_i}|f-f_{Q_y}|^p\,dm\\
 &\le \frac{V(y,r)}{V(x_i,t)}\Psi(r)M_f(y,r).
\end{align*}
Since $y\in B(x_i,r/4)$ and $r/t=\theta^{-1}$ is a fixed structural
constant, volume doubling gives
\begin{equation}\label{eq:vol-r-t}
 \frac{V(y,r)}{V(x_i,t)}\le C.
\end{equation}
Similarly,
$|a_j-f_{Q_y}|^p\le C\Psi(r)M_f(y,r)$.
Therefore
\begin{equation}\label{eq:coeff-pointwise}
 |a_i-a_j|^p\le C\Psi(r)M_f(y,r),
 \qquad y\in D_i.
\end{equation}
Averaging \eqref{eq:coeff-pointwise} over $D_i$ gives
\begin{equation}\label{eq:coeff-avg}
 |a_i-a_j|^p
 \le C\frac{\Psi(r)}{V(x_i,r/4)}
          \int_{D_i}M_f(y,r)\,dm(y).
\end{equation}

By \eqref{eq:part-energy},
\[
 \E(\psi_j)\le C_{\mathrm{cut}}
 \frac{V(x_j,t)}{\Psi(t)}.
\]
Since $j\in K_i$,
\[
 B(x_j,t)\subset B(x_i,(L+2)t)
 \subset B(x_i,r/4)=D_i,
\]
where the second inclusion follows from
$(L+2)t=r/16<r/4$.  Hence
$V(x_j,t)\le V(x_i,r/4)$.  Finally, because
$r=\theta^{-1}t$ and $\theta$ is fixed, the doubling property of $\Psi$
gives
\begin{equation}\label{eq:Psi-r-t}
 \Psi(r)\le C_\theta\Psi(t),
\end{equation}
where $C_\theta$ depends only on $C_\Psi$ and $A_c$.
Multiplying \eqref{eq:coeff-avg} by the energy bound for $\psi_j$ and
using these two observations proves \eqref{eq:coeff-energy}.
\end{proof}

\begin{lemma}\label{lem:approx}
For every ball $B=B(x,s)$ with $s>r$,
\begin{equation}\label{eq:approx}
 \int_B|f-P_rf|^p\,dm
 \le C\Psi(r)\int_{B(x,2s)}M_f(y,r)\,dm(y).
\end{equation}
\end{lemma}

\begin{proof}
By \eqref{eq:part-sum}, non-negativity, and convexity,
\begin{equation}\label{eq:convex-error}
 |f-P_rf|^p
 =\left|\sum_i\psi_i(f-a_i)\right|^p
 \le \sum_i\psi_i|f-a_i|^p.
\end{equation}
Fix $i$ and $y\in D_i$.  By \eqref{eq:part-supp} and
\eqref{eq:theta-geometry},
\[
 \suppm\psi_i\cup B_i\subset B(y,r)=Q_y.
\]
Therefore
\begin{align}
 \int_{\suppm\psi_i}|f-a_i|^p\,dm
 &\le 2^{p-1}\int_{Q_y}|f-f_{Q_y}|^p\,dm
 \nonumber\\
 &\quad
 +2^{p-1}m(\suppm\psi_i)|a_i-f_{Q_y}|^p.
 \label{eq:error-one-i-1}
\end{align}
The first term on the right is
$2^{p-1}\Psi(r)V(y,r)M_f(y,r)$.  For the second term, Jensen's
inequality and \eqref{eq:vol-r-t} yield
\[
 |a_i-f_{Q_y}|^p
 \le C\Psi(r)M_f(y,r).
\]
Furthermore,
\[
 m(\suppm\psi_i)\le V(x_i,Lt)
 \le C_{D,L}V(x_i,t)
 \le C'_{D,L}V(y,r),
\]
so \eqref{eq:error-one-i-1} gives
\begin{equation}\label{eq:error-one-i-point}
 \int_{\suppm\psi_i}|f-a_i|^p\,dm
 \le C\Psi(r)V(y,r)M_f(y,r).
\end{equation}
Since $y\in D_i$, volume doubling gives
$V(y,r)\le CV(x_i,r/4)=Cm(D_i)$.  Averaging
\eqref{eq:error-one-i-point} over $y\in D_i$ therefore yields
\begin{equation}\label{eq:error-one-i}
 \int_{\suppm\psi_i}|f-a_i|^p\,dm
 \le C\Psi(r)\int_{D_i}M_f(y,r)\,dm(y).
\end{equation}

Integrating \eqref{eq:convex-error} over $B$ and using
\eqref{eq:error-one-i},
\begin{equation}\label{eq:error-sum}
 \int_B|f-P_rf|^p\,dm
 \le C\Psi(r)
 \sum_{i:\suppm\psi_i\cap B\ne\varnothing}
 \int_{D_i}M_f(y,r)\,dm(y).
\end{equation}
The family $\{D_i\}$ has uniformly bounded multiplicity by
Lemma~\ref{lem:packing}, because $r/4=(4\theta)^{-1}t$ is a fixed
multiple of $t$.  If $\suppm\psi_i\cap B(x,s)\ne\varnothing$, then
$d(x_i,x)<s+Lt$; hence, for $y\in D_i$,
\[
 d(y,x)<s+Lt+r/4
 \le \left(1+L\theta+\frac14\right)s<2s.
\]
Thus all the relevant $D_i$ are contained in $B(x,2s)$.
Using the bounded multiplicity in \eqref{eq:error-sum} proves
\eqref{eq:approx}.
\end{proof}

\begin{lemma}\label{lem:energy}
For every ball $B=B(x,s)$ with $s>r$,
\begin{equation}\label{eq:energy}
 \Gp(P_rf)(A_PB)
 \le C\int_{B(x,(A_P+1)s)}M_f(y,r)\,dm(y).
\end{equation}
\end{lemma}

\begin{proof}
The balls $B_i$ cover $X$.  Let
\[
 I_B:=\{i:B_i\cap A_PB\ne\varnothing\}.
\]
By local finiteness and relative compactness, $I_B$ is finite, and
positivity of $\Gp(P_rf)$ gives
\begin{equation}\label{eq:energy-cover}
 \Gp(P_rf)(A_PB)
 \le\sum_{i\in I_B}\Gp(P_rf)(B_i).
\end{equation}

Fix $i\in I_B$.  Recall
$K_i=\{j:B_i\cap\suppm\psi_j\ne\varnothing\}$.  On $B_i$,
local finiteness and the partition identity give
\begin{equation}\label{eq:local-rewrite}
 P_rf-a_i
 =\sum_{j\in K_i}(a_j-a_i)\psi_j.
\end{equation}
The right-hand side is a finite sum belonging to $\F$.  By the definition of the
localized energy measure and Lemma~\ref{lem:constants},
\begin{equation}\label{eq:Gamma-local-rewrite}
 \Gp(P_rf)(B_i)
 =\Gp(P_rf-a_i)(B_i)
 =\Gp\left(\sum_{j\in K_i}(a_j-a_i)\psi_j\right)(B_i).
\end{equation}
Using the seminorm property \ref{G2}, followed by
$(\sum_{j=1}^N b_j)^p\le N^{p-1}\sum_{j=1}^N b_j^p$, we obtain
\begin{align}
 \Gp(P_rf)(B_i)
 &\le (\#K_i)^{p-1}
 \sum_{j\in K_i}|a_j-a_i|^p\Gp(\psi_j)(B_i)
 \nonumber\\
 &\le N_1^{p-1}
 \sum_{j\in K_i}|a_j-a_i|^p\E(\psi_j).
 \label{eq:Gamma-local-bound}
\end{align}
Lemma~\ref{lem:coeff} and $\#K_i\le N_1$ give
\begin{equation}\label{eq:Gamma-one-i}
 \Gp(P_rf)(B_i)
 \le C\int_{D_i}M_f(y,r)\,dm(y).
\end{equation}

The family $\{D_i\}$ has uniformly bounded multiplicity.  Moreover, if
$i\in I_B$, choose $z\in B_i\cap B(x,A_Ps)$.  Then
$d(x_i,x)<A_Ps+t$, and for $y\in D_i$,
\[
 d(y,x)<A_Ps+t+r/4
 \le (A_P+\theta+1/4)s<(A_P+1)s.
\]
Thus
\[
 \bigcup_{i\in I_B}D_i\subset B(x,(A_P+1)s).
\]
Summing \eqref{eq:Gamma-one-i} in \eqref{eq:energy-cover} and using the
bounded multiplicity proves \eqref{eq:energy}.
\end{proof}

\subsection{Proof of the macroscopic inequality}\label{sec:main-proof}

The proof follows the decomposition into approximation error and
Poincar\'e oscillation used by
Hyt\"onen--Korte~\cite[proof of Theorem~3.4]{HytonenKorte}.  Lemmas~\ref{lem:approx} and~\ref{lem:energy} provide the corresponding
estimates in terms of the energy measure.

\begin{proof}[Proof of Theorem~\ref{thm:main}]
Fix $f\in L^p_{\mathrm{loc}}(X;m)$, $x\in X$, and $0<r<s$.  When
$X$ has finite diameter, the case $r>\diam X$ was settled above.  In the
remaining cases, let $u=P_rf$ be the approximation constructed in
Subsection~\ref{sec:approximation}.

Set $B=B(x,s)$.  We first compare the oscillation of $f$ with the
approximation error and the oscillation of $u$.  Since
\[
 |f_B-u_B|=|(f-u)_B|
 \le V(x,s)^{-1/p}\|f-u\|_{L^p(B)},
\]
we have
\[
 \|f-f_B\|_{L^p(B)}
 \le \|f-u\|_{L^p(B)}+\|u-u_B\|_{L^p(B)}
      +V(x,s)^{1/p}|u_B-f_B|
\]
and therefore
\begin{equation}\label{eq:osc-decomp}
 \int_B|f-f_B|^p\,dm
 \le C_p\left(
 \int_B|f-u|^p\,dm+
 \int_B|u-u_B|^p\,dm\right),
\end{equation}
where $C_p$ depends only on~$p$.

By Lemma~\ref{lem:approx},
\begin{equation}\label{eq:final-error}
 \int_B|f-u|^p\,dm
 \le C\Psi(r)\int_{B(x,2s)}M_f(y,r)\,dm(y).
\end{equation}
By Lemmas~\ref{lem:PI-local} and~\ref{lem:energy},
\begin{align}
 \int_B|u-u_B|^p\,dm
 &\le C_P\Psi(s)\Gp(u)(A_PB)\nonumber\\
 &\le C\Psi(s)
 \int_{B(x,(A_P+1)s)}M_f(y,r)\,dm(y).
 \label{eq:final-energy}
\end{align}
Since $r<s$ and $\Psi$ is increasing,
$\Psi(r)\le\Psi(s)$, and since $A_P\ge1$, $B(x,2s)\subset B(x,(A_P+1)s)$.  Combining
\eqref{eq:osc-decomp}--\eqref{eq:final-energy} and setting
\[
 A_M:=A_P+1,
\]
we obtain
\begin{equation}\label{eq:pre-final}
 \int_{B(x,s)}|f-f_{B(x,s)}|^p\,dm
 \le C\Psi(s)
 \int_{B(x,A_Ms)}M_f(y,r)\,dm(y).
\end{equation}
Divide by $\Psi(s)V(x,s)$.  Using volume doubling once more at the fixed
scale $A_M$,
\[
 \frac{V(x,A_Ms)}{V(x,s)}\le C_D(A_M),
\]
we obtain from \eqref{eq:pre-final}
\[
 M_f(x,s)
 \le CC_D(A_M)
 \avint_{B(x,A_Ms)}M_f(y,r)\,dm(y).
\]
This is \eqref{eq:mPI} with
$C_M=CC_D(A_M)$.

\end{proof}

\begin{remark}\label{rem:no-reverse}
The capacity bound, normal contractions, and the algebra property enter
only through the controlled partition of unity.  The remaining argument
uses volume doubling, $\mathrm{PI}(\Psi)$, the localized properties
\ref{G1}, \ref{G2}, \ref{G4}, invariance under addition of constants, and
fixed-ratio upper comparisons of~$\Psi$.
\end{remark}

\begin{remark}\label{rem:Yang-alternative}
An alternative proof of Theorem~\ref{thm:main} may be obtained by adapting Yang~\cite[Proposition~4.5]{Yang2026nonlocal}. We sketch the argument.   The result \cite[Proposition~4.5]{Yang2026nonlocal} bounds the oscillation on a ball by localized pairwise differences at an arbitrary smaller scale.  Applying it at scale~$r/2$ and averaging the pairwise differences over nearby centers yields~\eqref{eq:mPI}, by volume doubling and doubling of~$\Psi$.  Yang works on unbounded spaces with a chain condition and two-sided power control of~$\Psi$; however, for this estimate, his argument extends to our setting upon using Lemma~\ref{lem:partition}, and only fixed-ratio comparisons of~$\Psi$ are needed.  The passage from~$L^p(X;m)$ to~$L^p_{\mathrm{loc}}(X;m)$ follows by localization. We did not follow that route to keep the paper self-contained. 
\end{remark}

\subsection{Microscopic control}\label{sec:microscopic-control}

We say that the microscopic control condition \eqref{eq_MC} holds if there
exist $C_{MC}>0$ and $A_{MC}\ge1$ such that, for every ball $B(x_0,R)$
and every $f\in\F$,
\begin{equation*}\label{eq_MC}\tag*{$\text{MC}(\Psi)$}
 \liminf_{r\downarrow0}\int_{B(x_0,R)}M_f(x,r)\,dm(x)
 \le C_{MC}\Gp(f)(B(x_0,A_{MC}R)).
\end{equation*}

The following proposition relates the macroscopic inequality to the
energy measure of the given form.

\begin{proposition}\label{prop:microscopic-control}
Suppose that Assumption~\ref{ass:E} holds.
\begin{enumerate}[label=(\roman*)]
\item If \eqref{eq:VD} and \eqref{eq:PI} hold, then \eqref{eq_MC} holds
with $A_{MC}=2A_P$ and $C_{MC}=C_PC_D(2A_P)$.  In fact, the estimate
holds with $\limsup$ in place of $\liminf$.
\item If \eqref{eq:mPI} and \eqref{eq_MC} hold, then \eqref{eq:PI} holds
with $C_P'=C_MC_{MC}$ and $A_P'=A_MA_{MC}$, without assuming
\eqref{eq:VD} or \eqref{eq:cap}.
\end{enumerate}
Consequently, under Assumption~\ref{ass:E}, \eqref{eq:VD}, and
\eqref{eq:cap},
\[
 \eqref{eq:PI}
 \quad\Longleftrightarrow\quad
 \eqref{eq:mPI}\ \text{and}\ \eqref{eq_MC}.
\]
\end{proposition}

\begin{proof}[Proof of Proposition~\ref{prop:microscopic-control}]
Assume \eqref{eq:VD} and \eqref{eq:PI}.  Fix $B=B(x_0,R)$,
$f\in\F$, and $0<r<R$.  By \eqref{eq:PI},
\[
 M_f(x,r)\le C_P\frac{\Gp(f)(B(x,A_Pr))}{V(x,r)}.
\]
If $x\in B(z,A_Pr)$, then
\[
 V(z,A_Pr)\le V(x,2A_Pr)\le C_D(2A_P)V(x,r).
\]
Tonelli's theorem therefore gives
\begin{align*}
 \int_B M_f(x,r)\,dm(x)
 &\le C_P\int_X
 \left(\int_{B\cap B(z,A_Pr)}\frac{dm(x)}{V(x,r)}\right)d\Gp(f)(z)\\
 &\le C_PC_D(2A_P)\Gp(f)(B(x_0,R+A_Pr))\\
 &\le C_PC_D(2A_P)\Gp(f)(B(x_0,2A_PR)).
\end{align*}
Taking the upper limit as $r\downarrow0$ proves the stated stronger
estimate, and hence \eqref{eq_MC}.

Conversely, assume \eqref{eq:mPI} and \eqref{eq_MC}.  Fix $B=B(x_0,R)$
and $f\in\F$.  For every $0<r<R$, \eqref{eq:mPI} and $A_M\ge1$ give
\begin{align*}
 \int_B|f-f_B|^p\,dm
 &\le C_M\Psi(R)\frac{V(x_0,R)}{V(x_0,A_MR)}
       \int_{B(x_0,A_MR)}M_f(x,r)\,dm(x)\\
 &\le C_M\Psi(R)\int_{B(x_0,A_MR)}M_f(x,r)\,dm(x).
\end{align*}
Taking the lower limit as $r\downarrow0$ and applying \eqref{eq_MC}
on $B(x_0,A_MR)$ yields
\[
 \int_B|f-f_B|^p\,dm
 \le C_MC_{MC}\Psi(R)\Gp(f)(B(x_0,A_MA_{MC}R)).
\]
This is \eqref{eq:PI} with $C_P'=C_MC_{MC}$ and $A_P'=A_MA_{MC}$.
No doubling or capacity estimate is used in this implication.
The final equivalence follows from these two implications and
Theorem~\ref{thm:main}.
\end{proof}

\subsection{Two-scale comparison of the Korevaar--Schoen functionals}
\label{sec:two-scale}
We compare integrated ball oscillations with the Korevaar--Schoen functionals~\eqref{eq:J}, then integrate~\eqref{eq:mPI} to prove Theorem~\ref{thm:two-scale}.

For $f\in L^p(X;m)$ and $r>0$, set
\begin{equation}\label{eq:O}
 \mathcal O_r(f):=\int_XM_f(x,r)\,dm(x).
\end{equation}

\begin{lemma}\label{lem:comparison}
Assume \eqref{eq:VD} and \eqref{eq:Psi-doubling}.  There is $C\ge1$,
depending only on $p,C_D,C_\Psi$, such that for every
$f\in L^p(X;m)$ and $r>0$,
\begin{equation}\label{eq:O-to-J}
 \mathcal O_r(f)\le CJ_{2r}(f),
\end{equation}
and
\begin{equation}\label{eq:J-to-O}
 J_r(f)\le C\mathcal O_{2r}(f).
\end{equation}
\end{lemma}

\begin{proof}
Set
\[
 \mathcal D_r(f):=\frac1{\Psi(r)}\int_X\avint_{B(x,r)}\avint_{B(x,r)}
 |f(y)-f(z)|^p\,dm(z)\,dm(y)\,dm(x).
\]
Jensen's inequality and the inequality
$|a+b|^p\le 2^{p-1}(|a|^p+|b|^p)$ give
\begin{equation}\label{eq:O-D}
 \mathcal O_r(f)\le\mathcal D_r(f)\le2^p\mathcal O_r(f).
\end{equation}
By Tonelli's theorem,
\[
 \mathcal D_r(f)=\frac1{\Psi(r)}\iint_{X\times X}|f(y)-f(z)|^p
 H_r(y,z)\,dm(z)\,dm(y),
\]
where
\[
 H_r(y,z):=\int_X
 \frac{\mathbf1_{B(x,r)}(y)\mathbf1_{B(x,r)}(z)}{V(x,r)^2}\,dm(x).
\]
The kernel vanishes when $d(y,z)\ge2r$.  For $x\in B(y,r)$, volume
doubling gives $V(x,r)\asymp V(y,r)$, and hence
\[
 H_r(y,z)\lesssim\frac{\mathbf1_{B(y,2r)}(z)}{V(y,2r)}.
\]
Together with $\Psi(2r)\lesssim\Psi(r)$, this proves
$\mathcal D_r(f)\lesssim J_{2r}(f)$ and therefore \eqref{eq:O-to-J}.
Conversely, if $d(y,z)<r$, then every $x\in B(y,r/2)$ satisfies
$y,z\in B(x,2r)$, and fixed-scale volume comparison gives
$H_{2r}(y,z)\gtrsim V(y,r)^{-1}$.  Therefore
\[
 J_r(f)\lesssim\frac{\Psi(2r)}{\Psi(r)}\mathcal D_{2r}(f)
 \lesssim\mathcal O_{2r}(f),
\]
which proves \eqref{eq:J-to-O}.
\end{proof}

\begin{proof}[Proof of Theorem~\ref{thm:two-scale}]
Let $f\in L^p(X;m)$ and $0<r<s$.  Integrating \eqref{eq:mPI} in $x$ and
applying Tonelli's theorem gives
\[
 \mathcal O_s(f)\le C_M\int_XM_f(y,r)
 \left(\int_{B(y,A_Ms)}\frac{dm(x)}{V(x,A_Ms)}\right)dm(y).
\]
For $x\in B(y,A_Ms)$, volume doubling gives
$V(y,A_Ms)\lesssim V(x,A_Ms)$, so the inner integral is uniformly
bounded.  This proves
\begin{equation}\label{eq:O-WM}
 \mathcal O_s(f)\le C\mathcal O_r(f),
\end{equation}
with $C$ depending only on $C_D,C_M$, and $A_M$.  Lemma~\ref{lem:comparison},
applied at scales $s$ and $r/2$, now yields
\[
 J_s(f)\lesssim\mathcal O_{2s}(f)
 \lesssim\mathcal O_{r/2}(f)\lesssim J_r(f),
\]
which is \eqref{eq:J-WM}.  For each fixed $s>0$, let $r\downarrow0$ in
\eqref{eq:J-WM}, and then take the supremum over $s$ to obtain
\eqref{eq:WM}.
\end{proof}

\subsection{Connectedness and lower scaling}\label{sec:scaling}

We now prove Proposition~\ref{prop:automatic-LS}, which requires no
energy form, chain condition, or doubling assumption on~$\Psi$.
Yang~\cite[Proposition~2.1 and Remark~2.2]{Yang2026} proved the lower
$p$-scaling \eqref{eq:automatic-LS} directly from a chain condition,
volume doubling, $\mathrm{PI}(\Psi)$, and the capacity bound.
Proposition~\ref{prop:automatic-LS} derives the same conclusion from~\eqref{eq:mPI}, with connectedness also obtained as a consequence.  To our knowledge, this implication from the macroscopic Poincar\'e inequality is new in this generality.  The proof first establishes connectedness and then applies~\eqref{eq:mPI} to a Lipschitz tent function.

\begin{proof}[Proof of Proposition~\ref{prop:automatic-LS}]
By the metric-measure observation in
Subsection~\ref{sec:metric-framework}, $X$ is proper.
We first show that $X$ is connected.  If
$X=U\mathbin{\dot\cup}W$ were a nontrivial clopen separation, choose
$x\in U$, $y\in W$, and $s>d(x,y)$.  Put
\[
 g:=\mathbf1_{U\cap B(x,(A_M+2)s)}\in L^p(X;m).
\]
The sets
\[
 U\cap\overline{B(x,(A_M+1)s)}
 \quad\text{and}\quad
 W\cap\overline{B(x,(A_M+1)s)}
\]
are nonempty disjoint compact sets.  Let $\Delta>0$ be their distance.
For $0<r<\min\{s,\Delta/2\}$, we have
\[
 M_g(z,r)=0,\qquad z\in B(x,A_Ms).
\]
On the other hand, $B(x,s)$ meets both $U$ and $W$ in sets of positive
measure, by full support, so $M_g(x,s)>0$.  This contradicts
\eqref{eq:mPI}.  Thus $X$ is connected.

Fix $0<r<R<\diam X$.  Choose $x,y\in X$ with $d(x,y)>R$.
Since the image of $w\mapsto d(x,w)$ is connected, there is a point $z$
with $d(x,z)=R/2$.  Consider the compactly supported $1$-Lipschitz function
\[
 f_R(w):=\bigl(R/8-d(x,w)\bigr)_+\in L^p(X;m).
\]
On $E=B(x,R/32)$ one has $f_R\ge3R/32$, whereas $f_R=0$ on
$F=B(z,R/32)$.  Both balls lie in $B=B(x,R)$, and volume doubling gives
\begin{equation}\label{eq:two-ball-mass}
 m(E),m(F)\ge cV(x,R).
\end{equation}
If $(f_R)_B\le3R/64$, integrate over $E$; otherwise, integrate over $F$.
In either case, \eqref{eq:two-ball-mass} gives
\begin{equation}\label{eq:tent-oscillation-lower}
 M_{f_R}(x,R)\ge c\frac{R^p}{\Psi(R)}.
\end{equation}
On the other hand, the Lipschitz bound implies, for every $v\in X$,
\begin{equation}\label{eq:tent-oscillation-upper}
 M_{f_R}(v,r)\le C\frac{r^p}{\Psi(r)}.
\end{equation}
Indeed, the oscillation of $f_R$ on a ball of radius $r$ is at most $2r$.
Applying \eqref{eq:mPI} to $f_R$ at scales $r<R$ and using
\eqref{eq:tent-oscillation-lower}--\eqref{eq:tent-oscillation-upper}, we
obtain
\[
 \frac{R^p}{\Psi(R)}\le C\frac{r^p}{\Psi(r)},
\]
which is \eqref{eq:automatic-LS} for $R<\diam X$.  If
$\diam X<\infty$, let $R\uparrow\diam X$ and use continuity of $\Psi$.
The case $r=R$ is immediate.
\end{proof}

\begin{corollary}
\label{cor:forward-LS}
Under the hypotheses of Theorem~\ref{thm:main},
\[
 \frac{\Psi(r)}{\Psi(R)}
 \le C\left(\frac rR\right)^p
\]
whenever $0<r\le R<\infty$ and, if $\diam X<\infty$,
$R\le\diam X$.  The constant depends only on
$p,C_D,C_\Psi,C_P,A_P,C_c$, and $A_c$.
\end{corollary}

\begin{proof}
Combine Theorem~\ref{thm:main} with
Proposition~\ref{prop:automatic-LS}.
\end{proof}

\begin{remark}\label{rem:O-regular}
Under the hypotheses of Theorem~\ref{thm:main},
Corollary~\ref{cor:forward-LS} and \eqref{eq:Psi-fixed} give
\[
 C^{-1}\left(\frac Rr\right)^p
 \le\frac{\Psi(R)}{\Psi(r)}
 \le C_\Psi\left(\frac Rr\right)^{\log_2C_\Psi}
\]
whenever $0<r\le R<\infty$ and, if $\diam X<\infty$, $R\le\diam X$.  The
upper bound follows from
$C_\Psi(R/r)\le C_\Psi^{1+\log_2(R/r)}$.  Thus every scale function
admissible in the forward results is of two-sided power type, with lower
exponent $p$ and upper exponent $\log_2C_\Psi$; letting $r\downarrow0$
with $R$ fixed shows in particular that $C_\Psi\ge2^p$.  These bounds allow nonpower scales, including the tree and Laakso-type
examples of Subsection~\ref{sec:examples}.  For $\Psi(r)=r^\beta$, combining
Shimizu~\cite[Theorem~4.6]{ShimizuBesov} with
Baudoin~\cite[Theorem~4.1]{BaudoinKSS} also gives $\beta\ge p$;
Proposition~\ref{prop:automatic-LS} gives relative scaling for a general
$\Psi$.
\end{remark}

\subsection{Identification of the energy domain}\label{subsec:domain-id}

We recall the domain characterization of Shimizu~\cite[Corollaries~1.14 and~1.15]{ShimizuBesov}, whose argument applies under the hypotheses of Theorem~\ref{thm:main} with the adjustments described below.  It gives
\begin{equation}\label{eq:domain}
 \F=\{f\in L^p(X;m):\liminf_{r\downarrow0}J_r(f)<\infty\}=\B_\Psi,
\end{equation}
and
\begin{equation}\label{eq:energy-equiv}
 \E(f)\asymp\liminf_{r\downarrow0}J_r(f)\asymp\sup_{r>0}J_r(f),\qquad f\in\F.
\end{equation}
The comparison constants depend only on~$p,C_D,C_\Psi,C_P,A_P,C_c$, and~$A_c$.

Shimizu's framework allows noncompact spaces, singular energy measures, and nonpower scale functions.  His two-sided power control of~$\Psi$ follows here from doubling and Corollary~\ref{cor:forward-LS}, at scales not exceeding~$\diam X$; see Remark~\ref{rem:O-regular}.  Proposition~\ref{prop:automatic-LS} gives connectedness, and hence uniform perfectness, while~$\mathrm{PI}(\Psi)$ gives~$\E^{-1}(0)\subset\mathbb R\mathbf1$.  Closedness and the Clarkson inequalities give reflexivity of the energy domain~\cite[Proposition~3.13]{KajinoShimizuContraction}.

For the discrete convolution argument in~\cite[Lemmas~3.10--3.11 and proof of Corollary~1.15]{ShimizuBesov}, Lemma~\ref{lem:partition} supplies the required partition under our normal-contraction and algebra assumptions.  The localized seminorm property and~\ref{G4} permit its use with measure-theoretic supports; continuity of the partition functions is not needed for the approximation and weak-compactness argument.  If~$X$ is bounded, only small radii enter the domain identification, and~$\mathrm{PI}(\Psi)$ controls the supremum over the remaining radii.  We use~\eqref{eq:domain}--\eqref{eq:energy-equiv} without repeating the proof.  On the Sierpi\'nski carpet, the corresponding identification is also given in~\cite[Theorem~1.4]{MuruganShimizuCarpet}.

\section{Construction of a $p$-energy from the macroscopic inequality}
\label{sec:converse}

For the construction we assume the metric-measure hypotheses of
Theorem~\ref{thm:converse}, volume doubling, doubling of $\Psi$,
\eqref{eq:mPI}, and the core condition \eqref{eq:Core}.  We construct an
energy form from these data.

\subsection{The ball kernel and the Korevaar--Schoen energy}

For $r>0$, define
\begin{equation}\label{eq:kernel}
 k_r(x,y):=\frac{\mathbf1_{B(x,r)}(y)}{\Psi(r)V(x,r)}.
\end{equation}
The measurability observation in Subsection~\ref{sec:metric-framework}
shows that $k_r$ is Borel.  The standing volume assumptions give
$0<V(x,r)<\infty$.  The associated nonlocal functional is $J_r$, and the
kernel is asymptotically local with
$\delta(r)=r$.

By Theorem~\ref{thm:two-scale},
\[
 \sup_{s>0}J_s(f)\le C\liminf_{r\downarrow0}J_r(f),\qquad f\in\B_\Psi.
\]
This is the weak monotonicity condition in
Sasaya~\cite[Definition~8.1]{Sasaya2026}.
By Sasaya~\cite[Theorem~8.3]{Sasaya2026}, based on the construction of
Kajino--Shimizu~\cite[Theorems~6.3.8 and~6.4.2]{KajinoShimizuKS}, whose
hypotheses (local compactness and separability, a full-support Radon
measure, asymptotic locality of the kernel, and weak monotonicity) are
satisfied here, there exists a sequence $r_n\downarrow0$ for which
\[
 \widetilde{\E}_{\KS}(f):=\lim_{n\to\infty}J_{r_n}(f),
 \qquad f\in\B_\Psi,
\]
is well defined and satisfies
\begin{equation}\label{eq:KS-equiv}
 C^{-1}\sup_{r>0}J_r(f)
 \le\widetilde{\E}_{\KS}(f)
 \le C\liminf_{r\downarrow0}J_r(f).
\end{equation}
Under \eqref{eq:Core}, Sasaya's theorem shows that
\[
 \F_{\KS}:=\overline{\B_\Psi\cap C_c(X)}^{\,\|\cdot\|_{\B_\Psi}}
\]
and
\[
 \E_{\KS}:=\left.\widetilde{\E}_{\KS}\right|_{\F_{\KS}}
\]
form a regular strongly local $p$-energy with the full $p$-Clarkson
properties \cite[Theorem~8.3]{Sasaya2026}.  The generalized
$p$-contraction property follows from
\cite[Theorem~6.3.8(b)]{KajinoShimizuKS}.  The Markov and algebra properties
follow from Kajino--Shimizu
\cite[Proposition~6.2.3, items~(2) and~(4)]{KajinoShimizuKS}.  The
associated measure $\GpKS(u)$ is identified in
\cite[Corollary~8.4]{Sasaya2026}.

\begin{lemma}\label{lem:localized-comparison}
For a measurable set $U\subset X$ and $f\in L^p(X;m)$, define, for $r>0$,
\[
 J_r(f;U):=\frac1{\Psi(r)}\int_U\avint_{B(y,r)}|f(y)-f(z)|^p\,dm(z)\,dm(y),
\]
and, for $q>0$, put $U_q:=\{y\in X:d(y,U)<q\}$.  Then for every $q>0$,
\begin{equation}\label{eq:localized-comparison}
 \int_UM_f(x,q)\,dm(x)\le CJ_{2q}(f;U_q),
\end{equation}
where $C$ depends only on $p,C_D,C_\Psi$.
\end{lemma}

\begin{proof}
Restrict the outer variable in the kernel argument of
Lemma~\ref{lem:comparison} to $U$.
The resulting kernel vanishes unless $y\in U_q$ and $d(y,z)<2q$.
Volume comparison bounds it by
\[
 C\frac{\mathbf1_{U_q}(y)\mathbf1_{B(y,2q)}(z)}{V(y,2q)}.
\]
The doubling of $\Psi$ gives \eqref{eq:localized-comparison}.
\end{proof}

\subsection{Recovery of the Poincar\'e inequality}

\begin{proof}[Proof of Theorem~\ref{thm:converse}]
It remains to prove \eqref{eq:converse-PI}.  First let
$u\in\B_\Psi\cap C_c(X)$, fix $x\in X$, $s>0$, and put
$U=B(x,A_Ms)$, where $A_M$ is the dilation in \eqref{eq:mPI}.  Set
$q_n=r_n/2$.  For all sufficiently large $n$, $q_n<s$, hence
\begin{equation}\label{eq:converse-step1}
 M_u(x,s)
 \le\frac{C}{V(x,A_Ms)}\int_UM_u(y,q_n)\,dm(y)
 \le\frac{C}{V(x,A_Ms)}J_{r_n}(u;U_{q_n}),
\end{equation}
by \eqref{eq:mPI} and Lemma~\ref{lem:localized-comparison}.  For large $n$,
$U_{q_n}\subset B(x,(A_M+1)s)$.  Since this ball has compact closure, choose
$\phi\in C_c(X)$ with $0\le\phi\le1$,
\[
 \phi=1\text{ on }\overline{B(x,(A_M+1)s)},\qquad
 \operatorname{supp}\phi\subset B(x,(A_M+2)s).
\]
Then
\begin{equation}\label{eq:converse-step2}
 J_{r_n}(u;U_{q_n})
 \le\frac1{\Psi(r_n)}\int_X\phi(y)\avint_{B(y,r_n)}
 |u(y)-u(z)|^p\,dm(z)\,dm(y).
\end{equation}
Define the measures
\[
 d\nu_n^u(y):=\frac1{\Psi(r_n)}
 \avint_{B(y,r_n)}|u(y)-u(z)|^p\,dm(z)\,dm(y).
\]
Kajino--Shimizu~\cite[Theorem~6.4.2, formula~(6.4.7)]{KajinoShimizuKS}
show that, along the same sequence $(r_n)$ as in the construction of
$\widetilde{\E}_{\KS}$, the integrals $\int_X\varphi\,d\nu_n^u$ converge
whenever $u\in\B_\Psi\cap C_b(X)$, $\varphi\in\B_\Psi\cap L^\infty(X;m)$,
and one of $u,\varphi$ has compact support; for
$\varphi\in\B_\Psi\cap C_c(X)$, the limit is the integral of $\varphi$
against the $p$-energy measure of $u$.  Since $u\in\B_\Psi\cap C_c(X)$,
this gives
\[
 \int_X\varphi\,d\nu_n^u
 \longrightarrow
 \int_X\varphi\,d\GpKS(u),
 \qquad \varphi\in\B_\Psi\cap C_c(X),
\]
where the identification of the limiting measure with the canonical
energy measure $\GpKS(u)$ is \cite[Theorem~8.3 and Corollary~8.4]{Sasaya2026}.
By \eqref{eq:Core}, choose $\phi_\ell\in\B_\Psi\cap C_c(X)$ with
$\|\phi_\ell-\phi\|_\infty\to0$.  Since
\[
 \sup_n\nu_n^u(X)=\sup_nJ_{r_n}(u)<\infty,
 \qquad \GpKS(u)(X)=\E_{\KS}(u)<\infty,
\]
uniform approximation extends the preceding convergence to $\phi$.
Passing to the limit in \eqref{eq:converse-step1} and \eqref{eq:converse-step2}
gives
\[
 M_u(x,s)\le\frac{C}{V(x,A_Ms)}
 \GpKS(u)(B(x,(A_M+2)s)).
\]
Multiplying by $\Psi(s)V(x,s)$ and using
$V(x,s)\le V(x,A_Ms)$ yields
\begin{equation}\label{eq:converse-core-PI}
 \int_{B(x,s)}|u-u_{B(x,s)}|^p\,dm
 \le C\Psi(s)\GpKS(u)(B(x,(A_M+2)s)).
\end{equation}

Finally, let $u\in\F_{\KS}$ and choose
$u_j\in\B_\Psi\cap C_c(X)$ converging to $u$ in $\|\cdot\|_{\B_\Psi}$.
By \eqref{eq:KS-equiv}, $u_j\to u$ in $L^p$ and
$\E_{\KS}(u_j-u)\to0$.  The left-hand side of
\eqref{eq:converse-core-PI} converges because
\[
 \|(u_j-u)-(u_j-u)_{B(x,s)}\|_{L^p(B(x,s))}
 \le2\|u_j-u\|_{L^p(B(x,s))}.
\]
For any Borel set $E$, the seminorm property of the canonical energy measure
gives
\[
 \left|\GpKS(u_j)(E)^{1/p}-\GpKS(u)(E)^{1/p}\right|
 \le\GpKS(u_j-u)(E)^{1/p}
 \le\E_{\KS}(u_j-u)^{1/p}\to0.
\]
Thus \eqref{eq:converse-core-PI} passes to the limit and proves
\eqref{eq:converse-PI} with $A=A_M+2$.
\end{proof}

\subsection{The core condition and reconstruction}

The macroscopic inequality is monotone in the scale function.

\begin{remark}\label{rem:mPI-monotone}
Let $\Phi:[0,\infty)\to[0,\infty)$ be nondecreasing with $\Phi>0$ on
$(0,\infty)$, and write $M_f^{\Psi}$ for the quantity \eqref{eq:M}
computed with the scale $\Psi$.  If \eqref{eq:mPI} holds for $\Psi$, then
it holds for $\Psi\Phi$ with the same constants $C_M$ and $A_M$.  Indeed,
for $0<r<s$, since $\Phi(r)\le\Phi(s)$,
\[
 M_f^{\Psi\Phi}(x,s)=\frac{M_f^{\Psi}(x,s)}{\Phi(s)}
 \le C_M\avint_{B(x,A_Ms)}\frac{M_f^{\Psi}(y,r)}{\Phi(s)}\,dm(y)
 \le C_M\avint_{B(x,A_Ms)}M_f^{\Psi\Phi}(y,r)\,dm(y).
\]
Hence \eqref{eq:mPI} alone does not determine the scale: it only forces
$\Psi$ to grow at least like $r^p$, by
Proposition~\ref{prop:automatic-LS}.  In Theorem~\ref{thm:converse}, it is
the density condition \eqref{eq:Core} that excludes scale functions which
grow too fast.
\end{remark}

In particular, the condition \eqref{eq:Core} does not follow from
\eqref{eq:mPI} alone.  On $[0,1]$ with Lebesgue measure,
$\mathrm{mPI}(r^p)$ holds by Theorem~\ref{thm:main} applied to the
standard $p$-energy, and Remark~\ref{rem:mPI-monotone} gives
$\mathrm{mPI}(r^\beta)$ for every $\beta\ge p$.  For $\beta>p$, however,
the corresponding supercritical Korevaar--Schoen class contains only
constants, since $[0,1]$ is a connected Poincar\'e space; see
Kajino--Shimizu~\cite[Introduction]{KajinoShimizuKS} and
Baudoin~\cite{BaudoinKSS}.
Thus its continuous finite-energy core is not uniformly dense in
$C([0,1])$.

\begin{remark}
\label{rem:cutoff-core}
Note that the core condition~\eqref{eq:Core} holds if, for every $x\in X$ and every relatively compact open
neighborhood $U$ of $x$ there exists $\eta\in C_c(X)$ such that
\begin{equation}\label{eq:continuous-KS-cutoff}
\begin{aligned}
 0&\le \eta\le 1,\\
 \eta&=1\ \text{on an open neighborhood of }x,
 \qquad \operatorname{supp}\eta\subset U,\\
 &\sup_{r>0}J_r(\eta)<\infty.
\end{aligned}
\end{equation}
Indeed, for every $r>0$, Minkowski's inequality gives
\[
J_r(fg)^{1/p}
\le \|f\|_\infty J_r(g)^{1/p}
+\|g\|_\infty J_r(f)^{1/p},\qquad f,g\in\B_\Psi\cap C_c(X).
\]
Thus $\B_\Psi\cap C_c(X)$ is a real algebra.  The cutoffs make it vanish nowhere and separate points: for distinct $x,y$, choose $U$ with $y\notin U$, so that $\eta(x)=1$ and $\eta(y)=0$.  The locally compact real Stone--Weierstrass theorem gives~\eqref{eq:Core}.  Under the hypotheses of Theorem~\ref{thm:two-scale}, the last condition in~\eqref{eq:continuous-KS-cutoff} may be replaced by $\liminf_{r\downarrow0}J_r(\eta)<\infty$, by~\eqref{eq:WM}.
\end{remark}

\begin{corollary}
\label{cor:reconstruction}
Under the hypotheses of Theorem~\ref{thm:main}, the core condition
\eqref{eq:Core} holds and the energy in Theorem~\ref{thm:converse} may be
constructed on the original domain $\F$.  Moreover,
\[
 \E_{\KS}(f)\asymp\E(f),\qquad f\in\F.
\]
\end{corollary}

\begin{proof}
Equations~\eqref{eq:domain} and~\eqref{eq:energy-equiv} give $\B_\Psi=\F$ and equivalence of
$\|\cdot\|_{\B_\Psi}$ with the graph norm.  Regularity of the original energy
implies that $\F\cap C_c(X)$ is uniformly dense in $C_c(X)$ and graph-norm
dense in $\F$.  Hence \eqref{eq:Core} holds, $\F_{\KS}=\F$, and the energy
equivalence follows from~\eqref{eq:KS-equiv} and~\eqref{eq:energy-equiv}.
\end{proof}

\begin{remark}
\label{rem:converse-capacity}
Theorem~\ref{thm:converse} gives $\mathrm{PI}(\Psi)$ for the reconstructed
form, but does not assert $\mathrm{cap}(\Psi)_{\le}$.  The cutoff criterion
in Remark~\ref{rem:cutoff-core} gives finite reconstructed energy by
\eqref{eq:KS-equiv}, without the uniform bound
$V(x,R)/\Psi(R)$ required by~\eqref{eq:cap}.
\end{remark}

\section{Weak-type characterization of the energy domain}\label{sec:weak}

\subsection{The pairwise upper estimate}

The next lemma refines the usual dyadic Poincar\'e telescoping argument
by retaining the energy measure on each ball.  Compare Yang
\cite[Lemma~3.4]{Yang2026} and, for the point-to-mean step
\eqref{eq:point-to-mean}, Dai--Lin--Yang--Yuan--Zhang
\cite[arXiv v1, Lemma~2.5]{DLYYZ}.

\begin{lemma}\label{lem:pairwise-telescope}
Assume \eqref{eq:VD}, \eqref{eq:Psi-doubling}, \eqref{eq:PI}, and
$\mathrm{LS}(\beta)$ for some $\beta>0$, with constant $C_L$.  There are
$A\ge1$, $\gamma>0$, and $C<\infty$ such that, for every $f\in\F$ and for
$m\otimes m$-almost every $(x,y)$ with $R:=d(x,y)>0$,
\begin{equation}\label{eq:pairwise-telescope}
 |f(x)-f(y)|^p
 \le C\Psi(R)\{\mathcal S_f(x,R)+\mathcal S_f(y,R)\},
\end{equation}
where
\begin{equation}\label{eq:S-pairwise}
 \mathcal S_f(x,R)
 :=\sum_{j=0}^\infty 2^{-j\gamma}
 \frac{\Gp(f)(B(x,A\,2^{-j}R))}{V(x,2^{-j}R)}.
\end{equation}
One may take $\gamma=\beta/2$.  The constants depend only on
$p,C_D,C_\Psi,C_P,A_P,C_L$, and $\beta$.
\end{lemma}

\begin{proof}
Use the Lebesgue representative of $f$, which exists almost everywhere by
the differentiation theorem for doubling measures.  Fix a Lebesgue point
$x$, put $r_j=2^{-j}R$ and $B_j=B(x,r_j)$.  Volume doubling, H\"older's
inequality, and \eqref{eq:PI} give
\begin{align*}
 |f_{B_{j+1}}-f_{B_j}|
 &\le C\left(\avint_{B_j}|f-f_{B_j}|^p\,dm\right)^{1/p}\\
 &\le C\Psi(r_j)^{1/p}
 \left(\frac{\Gp(f)(A_PB_j)}{V(x,r_j)}\right)^{1/p}.
\end{align*}
Since $f_{B_j}\to f(x)$ and \eqref{eq:Psi-lower} gives
$\Psi(r_j)\le C_L2^{-j\beta}\Psi(R)$,
\[
 |f(x)-f_{B(x,R)}|
 \le C\Psi(R)^{1/p}
 \sum_{j=0}^\infty2^{-j\beta/p}
 \left(\frac{\Gp(f)(A_PB_j)}{V(x,r_j)}\right)^{1/p}.
\]
Apply H\"older's inequality to the last series after writing
$2^{-j\beta/p}=2^{-j\beta/(2p)}2^{-j\beta/(2p)}$.  This yields
\begin{equation}\label{eq:point-to-mean}
 |f(x)-f_{B(x,R)}|^p
 \le C\Psi(R)\sum_{j=0}^\infty2^{-j\beta/2}
 \frac{\Gp(f)(A_PB_j)}{V(x,r_j)}.
\end{equation}

For Lebesgue points $x,y$ at distance $R$, compare the two means through
$B(x,3R)$.  Applying \eqref{eq:PI} on this ball, using nearby-ball volume
comparison and the fixed-scale estimate \eqref{eq:Psi-fixed}, gives
\[
 |f_{B(x,R)}-f_{B(y,R)}|^p
 \le C\Psi(R)
 \frac{\Gp(f)(B(x,3A_P R))}{V(x,R)}.
\]
Combine this estimate with \eqref{eq:point-to-mean} at $x$ and $y$, enlarge
$A$ by a fixed factor, and use the $p$-triangle inequality.  This proves
\eqref{eq:pairwise-telescope} with $\gamma=\beta/2$.
\end{proof}

\begin{proposition}\label{prop:pairwise-upper}
Under the hypotheses of Lemma~\ref{lem:pairwise-telescope},
\begin{equation}\label{eq:pairwise-upper}
 \mathfrak W_\Psi(f)\le C\E(f),\qquad f\in\F.
\end{equation}
The constant depends only on
$p,C_D,C_\Psi,C_P,A_P,C_L$, and $\beta$.
\end{proposition}

\begin{proof}
The anisotropic product metric, dyadic index extraction, and $5$-covering
argument follow the scheme of Dai--Lin--Yang--Yuan--Zhang
\cite[arXiv v1, proof of Theorem~2.4, in particular (2.19)--(2.24)]{DLYYZ}; see also
the covering argument of Brezis--Van Schaftingen--Yung
\cite[Propositions~2.1--2.2]{BVSY2021}.  The essential modification is that
the disjoint first-coordinate balls allow us to sum the possibly singular
measure $\Gp(f)$ directly.

Fix a Borel representative of $f\in\F$ and $\lambda>0$, and write
$D_\lambda=\{K_{\Psi,f}>\lambda\}$.  By volume doubling,
\begin{equation}\label{eq:nearby-pair-volume}
 V(x,R)\asymp V(y,R)\qquad\text{when }R=d(x,y)>0.
\end{equation}
Lemma~\ref{lem:pairwise-telescope} shows, up to a null set, that
$D_\lambda$ is contained in the union of the two sets
\[
 \mathcal S_f(x,R)>c\lambda V(x,R),
 \qquad
 \mathcal S_f(y,R)>c\lambda V(y,R).
\]
Interchanging $x$ and $y$ gives the same estimate for the second set, so
it suffices to consider the first.

Put $\delta=\gamma/2$.  If
$\mathcal S_f(x,R)>c\lambda V(x,R)$, then for some $j\ge0$,
\begin{equation}\label{eq:active-pair-ball}
 \Gp(f)(B(x,A\,2^{-j}R))
 >c_1\lambda2^{j\delta}
 V(x,2^{-j}R)V(x,R).
\end{equation}
Indeed, if \eqref{eq:active-pair-ball} failed for every $j$, then
\eqref{eq:S-pairwise} would give
$\mathcal S_f(x,R)\le c\lambda V(x,R)$ once $c_1$ is chosen sufficiently
small, since $\sum_{j\ge0}2^{-j(\gamma-\delta)}<\infty$.

The measurability observation in Subsection~\ref{sec:metric-framework},
applied to $\Gp(f)$ and $m$, shows that the active-pair sets below are
Borel.

For $j\ge0$ let $E_{\lambda,j}$ be the set of pairs $(x,y)$, with
$R=d(x,y)>0$, for which \eqref{eq:active-pair-ball} holds.  We claim that
\begin{equation}\label{eq:fixed-j-pair}
 (m\otimes m)(E_{\lambda,j})
 \le\frac{C2^{-j\delta}}\lambda\,\E(f).
\end{equation}
To prove the claim, fix $o\in X$ and first restrict $x,y$ to $B(o,N)$.
Choose $C_0>\max\{A,1\}$ and equip $X\times X$ with the metric
\[
 \rho_j((u,v),(u',v'))
 :=\max\{2^jA^{-1}d(u,u'),\,C_0^{-1}d(v,v')\}.
\]
For every active pair $(x,y)$ let $\mathcal B_{x,y}$ be the
$\rho_j$-ball with center $(x,x)$ and radius $R=d(x,y)$.  Thus
\[
 \mathcal B_{x,y}
 =B(x,A\,2^{-j}R)\times B(x,C_0R).
\]
The active pair $(x,y)$ belongs to this product ball, so these balls cover
the active set.  Their radii are bounded by $2N$.
The countable $5$-covering lemma in the separable metric space
$(X\times X,\rho_j)$ gives a countable disjoint
subfamily $\{\mathcal B_i\}$ whose fivefold dilates cover all the active
pairs in $B(o,N)^2$; see \cite[Theorem~1.2]{Heinonen2001}.
Fixed-ratio volume comparison and \eqref{eq:active-pair-ball} imply
\[
 (m\otimes m)(5\mathcal B_i)
 \le C V(x_i,2^{-j}R_i)V(x_i,R_i)
 \le\frac{C2^{-j\delta}}\lambda\,
 \Gp(f)(B(x_i,A\,2^{-j}R_i)).
\]
The first-coordinate balls $B(x_i,A\,2^{-j}R_i)$ are pairwise disjoint.
Indeed, each is contained in the corresponding second-coordinate ball,
and a common point in two first-coordinate balls would put its diagonal
pair in both product balls.  Consequently their $\Gp(f)$-masses sum to at
most $\E(f)$.  We obtain \eqref{eq:fixed-j-pair} on $B(o,N)^2$; monotone
convergence as $N\to\infty$ proves the claim.

Sum \eqref{eq:fixed-j-pair} over $j$.  The geometric series converges, and
the same argument applies to the set with $x,y$ interchanged.  Hence
\[
 (m\otimes m)(D_\lambda)\le\frac C\lambda\E(f).
\]
Taking the supremum over $\lambda$ proves \eqref{eq:pairwise-upper}.
\end{proof}

\subsection{Recovery from the high-level pairwise tail}

\begin{lemma}
\label{lem:pairwise-extraction}
Assume \eqref{eq:VD} and $\mathrm{LS}(\beta)$ for some $\beta>0$, with
constant $C_L$.  Let $f\in L^p(X;m)$ with $\mathfrak W_{\Psi,\infty}(f)<\infty$.
For $k>0$ put $f_k=(-k)\vee(f\wedge k)$.  Then for every metric ball
$U\subset X$,
\begin{equation}\label{eq:local-pairwise-extraction}
 \liminf_{r\downarrow0}J_r(f_k;U)\le C\mathfrak W_{\Psi,\infty}(f),
\end{equation}
where $J_r(\,\cdot\,;U)$ is the localized functional from
Lemma~\ref{lem:localized-comparison}.  The constant $C$ depends only on
$C_D,C_L$, and $\beta$, and is independent of $U,k$, and $f$.
\end{lemma}

\begin{proof}
Write $g=f_k$ and fix $0<\varepsilon<\delta$, with
$\delta<\diam X$ when $\diam X<\infty$.
The following logarithmic averaging over scales is the ball-kernel
analogue of the high-level tail arguments in
\cite[Section~4]{BSSVY2024} and
\cite[Theorem~1.5 and Section~2]{Poliakovsky2022}; compare also
\cite[Proposition~1.3]{HytonenYangYuanZhao2026} for a localized BBM
characterization on Poincar\'e spaces.  Let $d\Psi$ denote the
Lebesgue--Stieltjes measure of the continuous increasing function $\Psi$.
Tonelli's theorem gives
\begin{align}
 I_{\varepsilon,\delta}
 &:=\int_{(\varepsilon,\delta]}J_r(g;U)\,
 \frac{d\Psi(r)}{\Psi(r)} \notag\\
 &=\iint_{X\times X}L_{\varepsilon,\delta}(x,y)\,
 dm(y)\,dm(x),                                           \label{eq:I-pairwise}
\end{align}
where, with $a=\max\{\varepsilon,d(x,y)\}$,
\begin{equation}\label{eq:capped-kernel}
 L_{\varepsilon,\delta}(x,y)
 :=\mathbf1_U(x)\mathbf1_{\{d(x,y)<\delta\}}
 |g(x)-g(y)|^p
 \int_{(a,\delta]}
 \frac{d\Psi(r)}{\Psi(r)^2V(x,r)}.
\end{equation}
The integral is understood as zero when $a\ge\delta$.  Monotonicity of
$r\mapsto V(x,r)$ and the identity
\[
 \int_{(a,\delta]}\frac{d\Psi(r)}{\Psi(r)^2}
 =\frac1{\Psi(a)}-\frac1{\Psi(\delta)}
\]
give
\begin{equation}\label{eq:capped-dominated}
 0\le L_{\varepsilon,\delta}\le K_{\Psi,g}
 \le K_{\Psi,f}.
\end{equation}
The last inequality follows because truncation is $1$-Lipschitz.

The kernel $L_{\varepsilon,\delta}$ vanishes outside
$U\times U_\delta$, a set of finite $m\otimes m$-measure.  To quantify the
lower volume bound, choose $o$ and $R_0>0$ with $U\subset B(o,R_0)$.  Put
$\varepsilon_U:=\min\{1,2R_0\}$.  Then
$B(o,R_0)\subset B(x,2R_0)$ for $x\in U$, and iterating
\eqref{eq:VD} backwards from $2R_0$ to $\varepsilon$ gives
\[
 \inf_{x\in U}V(x,\varepsilon)\ge c_U\varepsilon^{Q_D},
 \qquad 0<\varepsilon<\varepsilon_U,
\]
where $Q_D:=\log_2 C_D$.  The constant $c_U>0$ absorbs the rounding of
the number of iterations.
Since $|g|\le k$, \eqref{eq:capped-kernel} therefore yields
\begin{equation}\label{eq:capped-Linfty}
 \|L_{\varepsilon,\delta}\|_\infty
 \le\frac{C_{U,k}}{\Psi(\varepsilon)\varepsilon^{Q_D}}
 =:M_\varepsilon.
\end{equation}

Fix $b>\mathfrak W_{\Psi,\infty}(f)$.  There is $0<\lambda_0<\infty$ such that
\[
 (m\otimes m)(\{K_{\Psi,f}>\lambda\})\le\frac b\lambda,
 \qquad \lambda\ge\lambda_0.
\]
Write $\log^+t:=\max\{\log t,0\}$ for $t>0$.  The layer-cake formula,
\eqref{eq:capped-dominated}, and
\eqref{eq:capped-Linfty} show that
\begin{equation}\label{eq:layer-pairwise}
 I_{\varepsilon,\delta}
 \le\lambda_0(m\otimes m)(U\times U_\delta)
 +b\log^+\frac{M_\varepsilon}{\lambda_0}.
\end{equation}
On the other hand,
\begin{equation}\label{eq:scale-average-pairwise}
 \left(\inf_{\varepsilon<r\le\delta}J_r(g;U)\right)
 \log\frac{\Psi(\delta)}{\Psi(\varepsilon)}
 \le I_{\varepsilon,\delta}.
\end{equation}
The lower scaling assumption implies, for fixed $\delta$,
\[
 \log\frac{\Psi(\delta)}{\Psi(\varepsilon)}
 \ge \beta\log\frac{\delta}{\varepsilon}-\log C_L.
\]
Together with \eqref{eq:capped-Linfty}, this gives
\begin{equation}\label{eq:log-ratio-pairwise}
 \limsup_{\varepsilon\downarrow0}
 \frac{\log^+(M_\varepsilon/\lambda_0)}
 {\log(\Psi(\delta)/\Psi(\varepsilon))}
 \le 1+\frac {Q_D}\beta.
\end{equation}
Continuity of $\Psi$ makes $d\Psi$ atomless, so no endpoint terms arise.
The dependence on $U$, $k$, $\lambda_0$, and the fixed $\delta$ occurs
only in the finite first term of \eqref{eq:layer-pairwise} or in additive
logarithmic terms.  These contributions vanish after division by
$\log(\Psi(\delta)/\Psi(\varepsilon))$ and passage to the limit.
Divide \eqref{eq:layer-pairwise} by the logarithm in
\eqref{eq:scale-average-pairwise} and let $\varepsilon\downarrow0$.  Since
\[
 \lim_{\varepsilon\downarrow0}
 \inf_{\varepsilon<r\le\delta}J_r(g;U)
 =\inf_{0<r\le\delta}J_r(g;U),
\]
letting next $\delta\downarrow0$ gives $\liminf_{r\downarrow0}J_r(g;U)$.
Finally, letting $b\downarrow\mathfrak W_{\Psi,\infty}(f)$ proves
\eqref{eq:local-pairwise-extraction}.
\end{proof}

\begin{lemma}\label{lem:local-global}
Assume the hypotheses of Theorem~\ref{thm:main}.  Let $f\in L^p(X;m)$,
$L<\infty$, and $f_k=(-k)\vee(f\wedge k)$.  If, for every ball $U$ and
every $k>0$,
\begin{equation}\label{eq:local-global-hyp}
 \liminf_{r\downarrow0}\int_U M_{f_k}(x,r)\,dm(x)\le L,
\end{equation}
then $f\in\F$ and $\E(f)\le CL$.
The constant $C$ depends only on
$p,C_D,C_\Psi,C_P,A_P,C_c$, and $A_c$.
\end{lemma}

\begin{proof}
Fix $x_0\in X$ and integrate \eqref{eq:mPI} over $B(x_0,R)$.
Tonelli's theorem and
nearby-ball volume comparison give, for $0<r<s$,
\[
 \int_{B(x_0,R)}M_{f_k}(x,s)\,dm(x)
 \le C\int_{B(x_0,R+A_Ms)}M_{f_k}(y,r)\,dm(y).
\]
Take the lower limit as $r\downarrow0$, use
\eqref{eq:local-global-hyp}, and let $R\uparrow\infty$.  Thus
$\mathcal O_s(f_k)\le CL$ for every $s,k>0$.
Lemma~\ref{lem:comparison} gives $\sup_sJ_s(f_k)\le CL$; Fatou's lemma in
\eqref{eq:J} gives the same bound for $f$.  The domain identification~\eqref{eq:domain} and the energy comparison~\eqref{eq:energy-equiv} complete the proof.
\end{proof}

\begin{proof}[Proof of Theorem~\ref{thm:pairwise-weak}]
By Corollary~\ref{cor:forward-LS}, $\Psi$ satisfies $\mathrm{LS}(p)$.
Proposition~\ref{prop:pairwise-upper} gives
\[
 \mathfrak W_{\Psi,\infty}(f)
 \le\mathfrak W_\Psi(f)\le C\E(f),\qquad f\in\F.
\]
Conversely, let $f\in L^p(X;m)$ with $\mathfrak W_{\Psi,\infty}(f)<\infty$.
Put $L=\mathfrak W_{\Psi,\infty}(f)$ and
$f_k=(-k)\vee(f\wedge k)$.  Fix a ball $U$ and a larger ball
$\widetilde U$ containing $U_r$ for all sufficiently small $r$.
The localized comparison
\eqref{eq:localized-comparison} and
Lemma~\ref{lem:pairwise-extraction} give
\[
 \liminf_{r\downarrow0}\int_UM_{f_k}(x,r)\,dm(x)
 \le C\liminf_{r\downarrow0}J_{2r}(f_k;\widetilde U)
 \le CL,
\]
uniformly in $U$ and $k$.  Lemma~\ref{lem:local-global} gives $f\in\F$ and
\begin{equation}\label{eq:pairwise-reverse}
 \E(f)\le CL=C\mathfrak W_{\Psi,\infty}(f).
\end{equation}
Combining these estimates proves~\eqref{eq:pairwise-domain} and~\eqref{eq:pairwise-energy}.
\end{proof}

\section*{Acknowledgments}
F.B.\ is partially supported by grant 10.46540/4283-00175B from the
Independent Research Fund Denmark, by the Villum Investigator grant
\emph{Stochastic Analysis in Aarhus}, and by the European Research Council
(ERC) under the European Union's Horizon Europe research and innovation
programme.  The ERC support is provided through the RanGe project
(Grant Agreement No.~101199772).

\section*{Declaration on the use of generative AI}

The authors are responsible for the mathematical ideas and the structure
of the paper.  ChatGPT (Astra) and Claude (Fable) were used to assist with proofreading, proof presentation, editorial
revision, and bibliographical research.  The authors reviewed and edited
the resulting content and take full responsibility for the publication.

\end{document}